\documentclass[12pt]{amsart}
\usepackage{extarrows}
\usepackage{amssymb}
\usepackage{mathrsfs}
\usepackage{stmaryrd}
\usepackage[shortlabels]{enumitem}
\usepackage{xcolor}
\usepackage[all]{xy}
\usepackage[margin=1in]{geometry} 
\usepackage{hyperref}
\hypersetup{bookmarksdepth=2}
\hypersetup{colorlinks=true}
\hypersetup{linkcolor=blue}
\hypersetup{citecolor=blue}
\hypersetup{urlcolor=blue}
\usepackage{eucal}

\numberwithin{equation}{section}
\newtheorem{theorem}[equation]{Theorem}
\newtheorem{proposition}[equation]{Proposition}
\newtheorem{lemma}[equation]{Lemma}
\newtheorem{corollary}[equation]{Corollary}
\theoremstyle{definition}
\newtheorem{remark}[equation]{Remark}
\newtheorem{question}[equation]{Question}

\newcommand{\bA}{\mathbf{A}}
\newcommand{\cA}{\mathcal{A}}

\newcommand{\cB}{\mathcal{B}}

\newcommand{\bC}{\mathbf{C}}
\newcommand{\cC}{\mathcal{C}}

\newcommand{\cD}{\mathcal{D}}

\newcommand{\rD}{\mathrm{D}}

\newcommand{\cE}{\mathcal{E}}

\newcommand{\sF}{\mathscr{F}}

\newcommand{\sG}{\mathscr{G}}

\newcommand{\rH}{\mathrm{H}}

\newcommand{\bI}{\mathbf{I}}

\newcommand{\cK}{\mathcal{K}}

\newcommand{\rK}{\mathrm{K}}

\newcommand{\rL}{\mathrm{L}}

\newcommand{\bN}{\mathbf{N}}

\newcommand{\bP}{\mathbf{P}}
\newcommand{\cP}{\mathcal{P}}

\newcommand{\PP}{\mathbb{P}}
\newcommand{\bQ}{\mathbf{Q}}
\newcommand{\cQ}{\mathcal{Q}}

\newcommand{\fR}{\mathfrak{R}}
\newcommand{\rR}{\mathrm{R}}

\newcommand{\fS}{\mathfrak{S}}

\newcommand{\bT}{\mathbf{T}}

\newcommand{\VV}{\mathbb{V}}

\newcommand{\bZ}{\mathbf{Z}}

\newcommand{\fa}{\mathfrak{a}}

\newcommand{\fb}{\mathfrak{b}}

\newcommand{\fh}{\mathfrak{h}}

\newcommand{\fm}{\mathfrak{m}}

\newcommand{\fp}{\mathfrak{p}}

\newcommand{\fq}{\mathfrak{q}}

\let\ol\overline

\DeclareMathOperator{\im}{im} 
\DeclareMathOperator{\End}{End}

\DeclareMathOperator{\Sym}{Sym}

\DeclareMathOperator{\Tor}{Tor}
\DeclareMathOperator{\Mod}{Mod}
\DeclareMathOperator{\Ind}{Ind}
\DeclareMathOperator{\Hom}{Hom}
\DeclareMathOperator{\Ext}{Ext}
\DeclareMathOperator{\Rep}{Rep}
\DeclareMathOperator{\Irr}{Irr}
\DeclareMathOperator{\soc}{soc}
\DeclareMathOperator{\avg}{avg}
\DeclareMathOperator{\rad}{rad}
\DeclareMathOperator{\uHom}{\underline{Hom}}

\DeclareMathOperator{\chr}{char}

\renewcommand{\phi}{\varphi}
\newcommand{\id}{\mathrm{id}}
\newcommand{\op}{\mathrm{op}}
\renewcommand{\Vec}{\mathrm{Vec}}
\newcommand{\tors}{\mathrm{tors}}
\newcommand{\gen}{\mathrm{gen}}
\newcommand{\FI}{\mathbf{FI}}
\newcommand{\WI}{\mathbf{WI}}

\newcommand{\arxiv}[1]{\href{http://arxiv.org/abs/#1}{{\tiny\tt arXiv:#1}}}
\newcommand{\DOI}[1]{\href{http://doi.org/#1}{\color{purple}{\tiny\tt DOI:#1}}}
\newcommand{\defn}[1]{\emph{#1}}

\title[Symmetric modules over the infinite polynomial ring I]{Symmetric modules over the infinite \\ polynomial ring I: nilpotent quotients}

\author{Rohit Nagpal}
\email{\href{mailto:rohitna@gmail.com}{rohitna@gmail.com}}

\author{Andrew Snowden}
\address{Department of Mathematics, University of Michigan, Ann Arbor, MI}
\email{\href{mailto:asnowden@umich.edu}{asnowden@umich.edu}}
\urladdr{\url{http://www-personal.umich.edu/~asnowden/}}
\thanks{AS was supported by NSF grant DMS-2301871.}

\author{Teresa Yu}
\address{Department of Mathematics, University of Michigan, Ann Arbor, MI}
\email{\href{mailto:twyu@umich.edu}{twyu@umich.edu}}
\urladdr{\url{https://sites.google.com/view/teresayu}}
\thanks{TY was supported by NSF DGE-2241144 and DMS-1840234.}

\date{August 6, 2025}

\begin{document}

\begin{abstract}
Cohen proved that the infinite variable polynomial ring $R=k[x_1,x_2,\ldots]$ is noetherian with respect to the action of the infinite symmetric group $\fS$. The first two authors began a program to understand the $\fS$-equivariant algebra of $R$ in detail. In previous work, they classified the $\fS$-prime ideals of $R$. An important example of an $\fS$-prime is the ideal $\fh_s$ generated by $(s+1)$st powers of the variables. In this paper, we study the category of $R/\fh_s$-modules. We obtain a number of results, and mention just three here: (a) we determine the Grothendieck group of the category; (b) we show that the Krull--Gabriel dimension is $s$; and (c) we obtain generators for the derived category. This paper will play a key role in subsequent work where we study general modules.
\end{abstract}

\maketitle
\tableofcontents

\section{Introduction}

Let $R=k[x_1, x_2, \ldots]$ be the infinite variable polynomial ring over the field $k$, equipped with its natural action of the infinite symmetric group $\fS$. Cohen \cite{Cohen} proved that $R$ is \emph{$\fS$-noetherian}, that is, the ascending chain condition holds for $\fS$-stable ideals. There has been much interest in Cohen's theorem in recent years; see, for instance, \cite{AschenbrennerHillar, CDDEF, DraismaEggermont, DEKL, DraismaKuttler, GunturkunNagel, HillarSullivant, KLS, LNNR, LNNR2, MarajNagel, NagelRomer, NagelRomer2}. An exposition of some of the main ideas in this area is provided in \cite{DraismaNotes}. Cohen's theorem suggests that there should be a good theory of $\fS$-equivariant algebra over $R$. In two previous papers \cite{svar, sideals}, the first two authors studied the $\fS$-ideals of $R$ in great detail. In this paper, and subsequent ones, we study $\fS$-equivariant modules over $R$, and determine basic structural features of this category. We note that while this paper continues the program initiated in \cite{svar, sideals}, it can be read independently of those papers.

\subsection{The general strategy}

Before stating our results, we outline our general strategy for studying equivariant $R$-modules.

An $\fS$-ideal $\fp$ of $R$ is called \defn{$\fS$-prime} if $\fa \fb \subset \fp$ implies $\fa \subset \fp$ or $\fb \subset \fp$, for $\fS$-ideals $\fa$ and $\fb$ of $R$. Clearly, an $\fS$-stable prime ideal of $R$ is an $\fS$-prime. However, there are $\fS$-primes that are not prime. The simplest example is the ideal $\fh_s$ generated by $x_i^{s+1}$ for all $i \ge 1$: this is $\fS$-prime for all $s \ge 0$, but only prime for $s=0$. This example is particularly surprising since $\fh_s$ is not even radical for $s>0$; it is therefore difficult to give a geometric interpretation to this ideal.

The main result of \cite{sideals} is a complete classification of the $\fS$-primes of $R$. We will not need the statement of the theorem in this paper, but we wish to highlight an important feature of the classification. Roughly speaking, a general $\fS$-prime is a kind of combination of ideals like $\fh_s$ and a radical $\fS$-prime. Since radical $\fS$-primes can be understood well geometrically, understanding the $\fh_s$ ideals is particularly important.

By a \emph{module} over the ring $R$ (and related rings), we will always mean one that is $\fS$-equivariant, and where the action of $\fS$ is smooth (see \S \ref{s:sinf}). The general goal of this paper (and the ensuing ones) is to understand the category of $R$-modules. If $\fp$ is an $\fS$-prime then one can take the category $\Mod_{R/\fp}$ of $R/\fp$-modules and form the Serre quotient by the subcategory of torison modules; this yields a category $\Mod_{R/\fp}^{\gen}$ that we call the \defn{generic category}. The basic idea of our approach is that if one understands the generic categories of all $\fS$-primes well enough, one should be able to assemble this into an understanding of $\Mod_R$ itself. In this paper, we focus on the generic categories of the $\fh_s$ ideals. In subsequent papers, we show how one can use this to understand general $\fS$-primes, and then show that this indeed allows us to understand all $R$-modules.

\subsection{The generic category}

As outlined above, the main goal of this paper is to understand the generic category of $R/\fh_s$-modules. To simplify exposition, we assume here that our field has characteristic~0, though in the body of the paper we do not require this.

There are two particularly important families of modules we must introduce. First, for $n \ge 0$, we let $\cP_{s,n}$ be the free $R/\fh_s$-module with basis $e_{i_1, \ldots, i_n}$ where $i_1, \ldots, i_n$ are distinct indices $\ge 1$. The finite symmetric group $\fS_n$ acts on $\cP_{s,n}$ by permuting the indices, and since we are in characteristic~0, we can decompose under this action: for a partition $\lambda$ of $n$, we let $\cP_{s,\lambda}$ be the $S^{\lambda}$-multiplicity space of $\cP_{n,s}$, where $S^{\lambda}$ is the Specht module. Next, we define $\cQ_{s,n}$ to be the quotient of $\cP_{s,n}$ where we kill the elements $x_i e_{i_1, \ldots, i_n}$, whenever $i$ is one of the indices $i_1, \ldots, i_n$. Again, we have an action of $\fS_n$, and we let $\cQ_{s,\lambda}$ be the $S^{\lambda}$-multiplicity space. We put a bar on any of these modules (e.g., $\ol{\cP}_{s,\lambda}$) to denote the object of the generic category they define.

The following three theorems encapsulate many of the main results of this paper. The first theorem is certainly the most fundamental:

\begin{theorem}
The category $\Mod_{R/\fh_s}^{\gen}$ is locally of finite length, i.e., any finitely generated object has finite length.
\end{theorem}

The next theorem explains the roles of the special objects we have defined.

\begin{theorem}
We have the following:
\begin{enumerate}
\item The simple objects of $\Mod_{R/\fh_s}^{\gen}$ are parametrized by partitions, via $\lambda \mapsto \soc(\ol{\cP}_{s,\lambda})$.
\item The $\ol{\cP}_{s,\lambda}$'s are exactly the indecomposable injective objects of $\Mod_{R/\fh_s}^{\gen}$.
\item The modules $\ol{\cQ}_{s,\lambda}$ generate $\rD^b_{\rm fg}(\Mod_{R/\fh_s}^{\gen})$ as a triangulated category.
\end{enumerate}
\end{theorem}

We refer to \S \ref{ss:dergen} for the notation and terminology used in (c). We note that, for $s \ge 1$, there are objects of $\Mod_{R/\fh_s}^{\gen}$ with infinite injective dimension, and so the $\ol{\cP}_{s,\lambda}$'s cannot generate $\rD^b_{\rm fg}(\Mod_{R/\fh_s}^{\gen})$; see Remark~\ref{rmk:P-dont-gen}. Also, while analogs of (a) and (c) remain vaild in positive characteristic, (b) is specific to characteristic~0.

We next turn to the Grothendieck group. Let $K$ be the Grothendieck group\footnote{Whenever we speak of Grothendieck groups, we will always mean the Grothendieck group of the category of finitely generated objects.} of the $\Mod_{R/\fh_s}^{\gen}$, and let $\Lambda$ be the Grothendieck group of $\Rep(\fS)$, which is isomorphic to the ring of symmetric functions. Since $\Mod_{R/\fh_s}^{\gen}$ is a module category for $\Rep(\fS)$, it follows that $K$ is naturally a $\Lambda$-module.

\begin{theorem}
We have the following:
\begin{enumerate}
\item The classes $[\ol{\cQ}_{s,\lambda}]$ form a $\bZ$-basis of $K$.
\item The classes $[\ol{\cP}_{s,\lambda}]$ form a $\bQ$-basis of $\bQ \otimes K$.
\item As a $(\bQ \otimes \Lambda)$-module, $\bQ \otimes K$ is free of rank one with basis $[\ol{\cP}_{s,0}]$.
\end{enumerate}
\end{theorem}

We note that the $\cP$'s do not span $K$ integrally, and $K$ is not free of rank one integrally. We do not have a good description of the $\Lambda$-module structure on $K$ at the integral level.

We also give a quite explicit description of the category $\Mod_{R/\fh_s}^{\gen}$ in \S \ref{ss:char0}, though this is specific to characteristic~0.

\subsection{General modules}

While we do not study general $R$-modules in this paper, we do prove some results about (non-generic) $R/\fh_s$-modules. We mention a few.

\begin{theorem}
The Krull--Gabriel dimension of the category $\Mod_{R/\fh_s}$ is $s$.
\end{theorem}

\begin{theorem}
The derived category $\rD^b_{\rm fg}(\Mod_{R/\fh_s})$ is generated by the objects $\cQ_{r,\lambda}$ for $0 \le r \le s$ and $\lambda$ a partition.
\end{theorem}

\begin{theorem}
The Grothendieck group $\bQ \otimes \rK(\Mod_{R/\fh_s})$ is a free $(\bQ \otimes \Lambda)$-module of rank $s+1$ with basis $[R/\fh_r]$ for $0 \le r \le s$.
\end{theorem}

An important technical point in the above theorems is that the section functor $S \colon \Mod_{R/\fh_s}^{\gen} \to \Mod_{R/\fh_s}$ has nice finiteness properties; see \S \ref{ss:semi-orth}.

\subsection{The key idea}

Let $\FI$ be the category of finite sets and injections. An \defn{$\FI$-module} is a functor from the category $\FI$ to the category of vector spaces. These objects were introduced by Church, Ellenberg, and Farb \cite{fimodule}, and studied in depth (in characteristic~0) in \cite{SSglI}. The most important idea in this paper is to relate $R/\fh_s$-modules to $\FI$-modules. Since $\FI$-modules are well-understood, this yields important information about $R/\fh_s$-modules.

Here is the basic idea. Suppose that $M$ is an $\FI$-module. Consider the vector space
\begin{displaymath}
N = \bigoplus_S M(S),
\end{displaymath}
where the sum is taken over all finite subsets $S$ of $\{1,2,\ldots\}$. The space $N$ clearly has an action of $\fS$. Moreover, we can give $N$ the structure of an $R/\fh_1$-module, in the following manner. Let $m \in M(S)$ be given. If $i \not\in S$ then we define $x_i m$ to be $f_*(m)$, where $f \colon S \to S \cup \{i\}$ is the inclusion and $f_* \colon M(S) \to M(S \cup \{i\})$ is the given map. If $i \in S$ then we define $x_i m=0$. It is not difficult to verify that $N$ is indeed an $R/\fh_1$-module.

Using an extension of this idea, we define a functor
\begin{displaymath}
\Phi_s \colon \Mod_{\FI} \to \Mod_{R/\fh_s}
\end{displaymath}
for any $s \ge 1$; see \S \ref{s:Phi}. This functor takes the basic projective $\FI$-modules to the $\cQ_s$-modules, and the basic torsion $\FI$-modules to the $\cQ_{s-1}$-modules. The functor $\Phi_s$ induces a functor on generic categories
\begin{displaymath}
\Psi_s \colon \Rep(\fS) \to \Mod^{\gen}_{R/\fh_s},
\end{displaymath}
where here we have identified $\Mod_{\FI}^{\gen}$ with $\Rep(\fS)$. This functor is studied in depth in \S \ref{s:Psi}, and is the primary tool used in our analysis of $\Mod_{R/\fh_s}^{\gen}$.

\subsection{Further comments}

We make a few additional comments on this work.
\begin{enumerate}
\item There have been several other generic categories that have been studied in the setting of equivariant commuative algebra; see, e.g., \cite{Ganapathy, SSglI, SSglII, symc1sp, sym2noeth, periplectic, FInmod}. This paper is perhaps the first case where detailed structural information about a generic category has been obtained without leveraging geometry.
\item In classical commutative algebra, higher Tor groups are torsion. This has remained true in equivariant commutative algebra, in the cases investigated so far (that we are aware of). However, we will see (\S \ref{ss:gentor}) that it is not true for $R/\fh_s$. Perhaps this is related to lack of geometry.
\item One might hope that $\Mod_{R/\fh_s}^{\gen}$ is a (semi-infinite) highest weight category, at least in characteristic~0, with the $\ol{\cQ}_{s,\lambda}$ being the co-standard modules. While it is true that there is an upper triangular change of basis between the $\ol{\cQ}_{s,\lambda}$'s and the simple modules, there is not an upper triangular relation between the injective modules $\ol{\cP}_{s,\lambda}$ and the $\cQ_{s,\lambda}$'s. For example, $\ol{\cP}_{s,(1)}$ has a filtration of length $s+1$ where each graded piece is $\ol{\cQ}_{s,(1)}$ (Proposition~\ref{prop:PQfilt}). Thus we do not have a highest weight structure.
\end{enumerate}

\subsection{Notation}

We list the most important notation:
\begin{description}[align=right,labelwidth=2cm,leftmargin=!]
\item[ $k$ ] the coefficient field
\item[ $S^{\lambda}$ ] the Specht module for $\fS_n$, where $\lambda$ is a partition of $n$
\item[ $\bP_n$ ] the $n$th principal projective $\FI$-module (\S \ref{ss:FI-proj})
\item[ $\fS$ ] the infinite symmetric group (\S \ref{ss:smooth-rep})
\item[ $\VV_n$ ] the representation of $\fS$ with basis $e_{i_1,\ldots,i_n}$, with $i_1, \ldots, i_n$ distinct (\S \ref{ss:Vrep})
\item[ $R$ ] the polynomial ring $k[x_i]_{i \ge 1}$ (\S \ref{ss:poly-ring})
\item[ $\fh_s$ ] the ideal $(x_i^{s+1})_{i \ge 1}$ of $R$
\item[ $\cP_{s,n}$ ] the module $R/\fh_s \otimes \VV_n$ (\S \ref{ss:Pmod})
\item[ $\cQ_{s,n}$ ] a quotient of the module $\cP_{s,n}$ (\S \ref{ss:Qmod})
\end{description}

\section{\texorpdfstring{$\FI$}{FI}-modules}\label{s:fimod}

In this section, we review the theory of $\FI$-modules. These modules will be of use to us in two ways. First, there is a close relationship between $\FI$-modules and the representation theory of $\fS$, and we will deduce some results about $\fS$ from known results about $\FI$-modules. And second, in \S \ref{s:Phi}, we construct a functor from the category of $\FI$-modules to the module category for $R/\fh_s$. This functor will play a key role in our analysis of $R/\fh_s$-modules. We note that an expository treatment of $\FI$-modules, in characteristic~0, can be found in the notes \cite{msri}.

\subsection{Basic definitions}

Let $\FI$ be the category of finite sets and injective morphisms. An \defn{$\FI$-module} is a functor from $\FI$ to the category of $k$-vector spaces. A \defn{map} of $\FI$-modules is a natural transformation of functors. We let $\Mod_{\FI}$ denote the category of $\FI$-modules. It is a $k$-linear Grothendieck abelian category.

Suppose $M$ is an $\FI$-module. We let $M_n$ denote the value of $M$ on the finite set $[n]$, which is a representation of the symmetric group $\fS_n$. The standard injection $[n] \to [n+1]$ induces a linear map $M_n \to M_{n+1}$, which is $\fS_n$-equivariant with respect to the standard inclusion $\fS_n \subset \fS_{n+1}$. This data completely describes $M$; that is, giving an $\FI$-module is equivalent to giving a sequence $(M_n)_{n \ge 0}$, where $M_n$ is a representation of $\fS_n$, equipped with transition maps $M_n \to M_{n+1}$ satisfying certain axioms; see \cite[Exercise~2.6]{msri}.

\subsection{Noetherianity}

Given an $\FI$-module $M$ and a collection $S$ of elements in various $M_n$'s, one can consider the $\FI$-submodule of $M$ \defn{generated} by $S$. We say that $M$ is \defn{finitely generated} if it is generated by a finite set of elements. One of the most important results about $\FI$-modules is the noetherianity property:

\begin{proposition}
An $\FI$-submodule of a finitely generated $\FI$-module is finitely generated.
\end{proposition}

This theorem was proved in \cite{deltamod, fimodule} in characteristic~0, and \cite{fimodule2} in general.

\subsection{Torsion modules}

Let $M$ be an $\FI$-module and let $x \in M_n$. We say that $x$ is \defn{torsion} if there is some injection $i \colon [n] \to [m]$ such that $i_*(x)=0$, where $i_*$ denotes the value of $M$ on $i$. The set of all torsion elements of $M$ forms an $\FI$-submodule of $M$ denoted $M_{\rm tors}$. We say that $M$ itself is \defn{torsion} if $M=M_{\rm tors}$. A finitely generated torsion module is supported in finitely many degrees and has finite length.

For a representation $V$ of $\fS_n$, we let $\bT(V)$ denote the $\FI$-module that is given by $V$ in degree $n$, and zero in all other degrees. We also let $\bT_n=\bT(k[\fS_n])$. The simple $\FI$-modules are exactly the $\bT(V)$ where $V$ is an irreducible $\fS_n$-representation. In particular, in characteristic~0, the simple $\FI$-modules are $\bT(S^{\lambda})$, where $S^{\lambda}$ denotes the Specht module attached to the partition $\lambda$.

\subsection{Projective modules} \label{ss:FI-proj}

Let $\bP_n$ be the $\FI$-module represented by the finite set $[n]$; that is, for a finite set $S$, we have
\begin{displaymath}
\bP_n(S) = k[\Hom_{\FI}(S, [n])].
\end{displaymath}
By Yoneda's lemma, for an arbitrary $\FI$-module $M$ we have
\begin{displaymath}
\Hom_{\FI}(\bP_n, M) = M_n.
\end{displaymath}
This shows that the functor $\Hom_{\FI}(\bP_n, -)$ is exact, and so $\bP_n$ is projective. We call $\bP_n$ the $n$th \defn{principal projective}.

The mapping property for $\bP_n$ shows that $\End_{\FI}(\bP_n)=k[\fS_n]$. In particular, in characteristic~0, we have an $\fS_n$-equivariant decomposition
\begin{displaymath}
\bP_n=\bigoplus_{\lambda \vdash n} S^{\lambda} \otimes \bP_{\lambda},
\end{displaymath}
where $\bP_{\lambda}$ is defined as the $S^{\lambda}$-multiplicity space in $\bP_n$; it is not difficult to see that the $\bP_{\lambda}$ are exactly the indecomposable projective $\FI$-modules.

One easily sees that an $\FI$-module is finitely generated if and only if it is a quotient of a finite direct sum of principal projective modules. In particular, if $M$ is a finitely generated $\FI$-module then there is a resolution $P_{\bullet} \to M$ where each $P_i$ is a finite sum of principal projectives. (This statement uses the noetherian theorem.)

\subsection{Induced modules}

The functor
\begin{displaymath}
\Mod_{\FI} \to \Rep(\fS_n), \qquad M \mapsto M_n
\end{displaymath}
has a left adjoint that we denote by $\bP(-)$. Thus if $V$ is a representation of $\fS_n$ then we have an associated $\FI$-module $\bP(V)$, characterized by the mapping property
\begin{displaymath}
\Hom_{\FI}(\bP(V), M) = \Hom_{\fS_n}(V, M_n).
\end{displaymath}
We have a natural identification
\begin{displaymath}
\bP(V) = (V \otimes \bP_n)_{\fS_n},
\end{displaymath}
where $(-)_{\fS_n}$ denotes $\fS_n$ co-invariants. We can also use invariants instead of co-invariants:

\begin{proposition} \label{prop:ind-inv}
We have a natural identification $\bP(V) = (V \otimes \bP_n)^{\fS_n}$.
\end{proposition}

\begin{proof}
For any $k[\fS_n]$-module $W$, we have a natural map
\begin{displaymath}
\avg \colon W_{\fS_n} \to W^{\fS_n}, \qquad \avg(x) = \sum_{g \in \fS_n} gx.
\end{displaymath}
This map is an isomorphism if $W$ is free, i.e., a sum of regular representations. If $M$ is an $\FI$-module with an action of $\fS_n$, then averaging defines a map of $\FI$-modules
\begin{displaymath}
\avg \colon M_{\fS_n} \to M^{\fS_n}.
\end{displaymath}
If $M_m$ is a free $k[\fS_n]$-module for all $m$ then this map is an isomorphism. We apply this with $M=V \otimes \bP_n$. Since $\bP_n([m])$ is a free $k[\fS_n]$-module for all $m$, we see that $V \otimes \bP_n([m])$ is as well, and so the result follows.
\end{proof}

We call $\FI$-modules of the form $\bP(V)$ \defn{induced}. Note that $\bP_n=\bP(k[\fS_n])$ and (in characteristic~0) $\bP_{\lambda}=\bP(S^{\lambda})$. From this, one sees that projective $\FI$-modules are (sums of) induced modules. Conversely, in characteristic~0, a representation $V$ of $\fS_n$ decomposes into Specht modules, and so the induced module $\bP(V)$ decomposes into a direct sum of $\bP_{\lambda}$'s; hence induced modules are projective. In positive characteristic, there are induced modules that are not projective.

\subsection{Semi-induced modules}

We say that an $\FI$-module $M$ is \defn{semi-induced} if there is a finite length filtration $0=F_0 \subset \cdots \subset F_n=M$ such that the graded pieces $F_i/F_{i-1}$ are induced $\FI$-modules. In characteristic~0, a semi-induced module is projective. In positive characteristic, there exist semi-induced modules that are not induced. Semi-induced modules admit the following elegant characterization:

\begin{proposition} \label{prop:FI-dersat}
A finitely generated $\FI$-module $M$ is semi-induced if and only if for every torsion $\FI$-module $T$ we have $\Ext^i_{\FI}(T, M)=0$ for all $i \ge 0$.
\end{proposition}

\begin{proof}
See \cite[Theorem~A.9]{Djament}. The result also follows from \cite{SSglI} in characteristic~0.
\end{proof}

\subsection{The shift theorem}

Let $M$ be an $\FI$-module. We define the \defn{shift} of $M$ to be the $\FI$-module $\Sigma M$ defined by $(\Sigma M)(S)=M(S \amalg \{\ast\})$, or, equivalently, $(\Sigma M)_n=M_{n+1}$, equipped with the natural transition maps. The shift operation defines an exact endofunctor of $\Mod_{\FI}$. We let $\Sigma^n$ denote its $n$-fold iterate. We have the following important result about this operation, due to the first author \cite[Theorem~A]{Nagpal1}.

\begin{proposition} \label{prop:shift}
Let $M$ be a finitely generated $\FI$-module. Then there exists $n \ge 0$ such that $\Sigma^n(M)$ is semi-induced.
\end{proposition}

If $M$ is any $\FI$-module then there exists a natural map $M \to \Sigma^n(M)$ whose kernel is torsion. Assuming $M$ is finitely generated, the theorem thus shows that there is a map from $M$ to a semi-induced module with torsion kernel. By iterating the above corollary, and using the fact that the cokernel of $M \to \Sigma^n(M)$ is smaller than $M$ (when measured by generation degree), we obtain another important corollary. This was first proved in general in \cite[Theorem~A]{Nagpal1}, but in characteristic~0 follows from the results of \cite{SSglI}.

\begin{corollary} \label{cor:semi-ind-res}
Let $M$ be a finitely generated $\FI$-module. Then there exists a finite length complex
\begin{displaymath}
0 \to M \to I^0 \to \cdots \to I^n \to 0
\end{displaymath}
where each $I^i$ is finitely generated and semi-induced, and all cohomology groups are torsion.
\end{corollary}

\subsection{The generic category} \label{ss:FI-gen}

Let $\Mod_{\FI}^{\tors}$ be the full subcategory of $\Mod_{\FI}$ spanned by torsion modules. This is a Serre subcategory. We define the \defn{generic category} of $\FI$-modules, denoted $\Mod_{\FI}^{\gen}$, to be the Serre quotient category $\Mod_{\FI}/\Mod_{\FI}^{\tors}$. We let
\begin{displaymath}
T \colon \Mod_{\FI} \to \Mod_{\FI}^{\gen}
\end{displaymath}
be the quotient functor. The category $\Mod_{\FI}^{\gen}$ is a Grothendieck abelian category, and the functor $T$ is exact and commutes with colimits. These claims follow from general properties of Serre quotients; see \cite[\S III]{Gabriel}. The most fundamental result about the generic category is the following:

\begin{proposition}
The category $\Mod_{\FI}^{\gen}$ is locally of finite length: if $M$ is a finitely generated $\FI$-module then $T(M)$ has finite length.
\end{proposition}

\begin{proof}
This was first prove in characteristic~0 in \cite[Corollary~2.2.6]{SSglI}. The general case follows from \cite[Theorem~2]{Djament}.
\end{proof}

For an abelian category $\cA$, let $\Irr(\cA)$ denote the set of isomorphism classes of simple objects in $\cA$. We have the following description of this for generic $\FI$-modules. We write $\ol{\bP}(V)$ for the image of $\bP(V)$ in the generic category.

\begin{proposition} \label{prop:FI-gen-simple}
If $V$ is a simple $k[\fS_n]$-module then $\soc(\ol{\bP}(V))$ is a simple object of $\Mod_{\FI}^{\gen}$. Moreover, this construction defines a bijection
\begin{displaymath}
\coprod_{n \ge 0} \Irr(\Rep(\fS_n)) \to \Irr(\Mod_{\FI}^{\gen}).
\end{displaymath}
In particular, if $\chr(k)=0$, then the simple generic $\FI$-modules are given by $\soc(\ol{\bP}_{\lambda})$ where $\lambda$ is a partition.
\end{proposition}

\begin{proof}
This was proved in characteristic~0 in \cite[Props~2.2.7 and~2.2.8]{SSglI}. In positive characteristic, it is a result of the first author \cite[\S 1.2]{Nagpal2}.
\end{proof}

For a locally noetherian abelian category $\cA$, let $\rK(\cA)$ be the Grothendieck group of the category of noetherian objects. Of course, $\rK(\Mod_{\FI}^{\gen})$ has a $\bZ$-basis consisting of the classes of simple objects. We now give a different basis, that is sometimes more useful.

\begin{proposition} \label{prop:gen-FI-K}
The map
\begin{displaymath}
\phi \colon \bigoplus_{n \ge 0} \rK(\Rep(\fS_n)) \to \rK(\Mod_{\FI}^{\gen}), \qquad [V] \mapsto [\ol{\bP}(V)]
\end{displaymath}
is an isomorphism. In particular, if $\chr(k)=0$, then the $[\ol{\bP}_{\lambda}]$ form a $\bZ$-basis of $\rK(\Mod_{\FI}^{\gen})$.
\end{proposition}

\begin{proof}
Corollary~\ref{cor:semi-ind-res} shows that $\phi$ is surjective. The more precise version of the corollary given in \cite[Theorem~2.7]{periodicity} shows that if $V$ is a simple representation of $\fS_n$ then $[\soc(\ol{\bP}(V))]$ is equal to $[\ol{\bP}(V)]$ plus terms of the form $[\ol{\bP}(W)]$, where $W$ is a representation of $\fS_m$ with $m<n$. Thus there is an upper-triangular change of basis between the classes of simples and the classes of induced modules, which completes the proof.
\end{proof}

\begin{remark} \label{rmk:FI-gen-inj}
In characteristic~0, the $T(\bP_{\lambda})$ are exactly the indecomposable injective objects of $\Mod_{\FI}^{\gen}$ \cite[Proposition~2.2.10]{SSglI}.
\end{remark}

\subsection{Saturation} \label{ss:fi-sat}

Let
\begin{displaymath}
S \colon \Mod_{\FI}^{\gen} \to \Mod_{\FI}
\end{displaymath}
be the right adjoint of $T$; this is called the \defn{section functor}. Since $S$ is a right adjoint to an exact functor, it follows that $S$ is left exact and preserves injective objects. If $N$ is an object of $\Mod_{\FI}^{\gen}$ then the natural map $N \to T(S(N))$ is an isomorphism. See \cite[\S III.2]{Gabriel} for general background on the section functor.

Let $M$ be an $\FI$-module. We call $S(T(M))$ the \defn{saturation} of $M$, and we say that $M$ is \defn{saturated} if the natural map $M \to S(T(M))$ is an isomorphism. An $\FI$-module is saturated if and only if $\Ext^i_{\FI}(T, M)=0$ for $i=0,1$ and all torsion modules $T$.

We say that $M$ is \defn{derived saturated} if the natural map $M \to \rR S(T(M))$ is an isomorphism; this simply means that $M$ is saturated and $\rR^i S(T(M))=0$ for all $i>0$. An $\FI$-module is derived saturated if and only if $\Ext^i_{\FI}(T, M)=0$ for all $i$ and all torsion modules $T$. Proposition~\ref{prop:FI-dersat} can thus be reformulated as follows: a finitely generated $\FI$-module is derived saturated if and only if it is semi-induced.

The following proposition gives some finiteness properties of derived saturation.

\begin{proposition} \label{prop:FI-sec}
Let $M$ be a finitely generated $\FI$-module, and let
\begin{displaymath}
0 \to M \to I^0 \to \cdots \to I^n \to 0
\end{displaymath}
be a complex with torsion cohomology where each $I^i$ is finitely generated and semi-induced (see Corollary~\ref{cor:semi-ind-res}).
\begin{enumerate}
\item $\rR ST(M)$ is computed by the complex $I^{\bullet}$.
\item $ST(M)=\ker(I^0 \to I^1)$ is finitely generated.
\item $\rR^i ST(M)$ has finite length for $i>0$.
\end{enumerate}
\end{proposition}

\begin{proof}
Since $T$ is exact and kills torsion modules, $T(M) \to T(I^{\bullet})$ is a resolution. Since semi-induced modules are derived saturated, $T(I^i)$ is acyclic for $S$. Thus this resolution can be used to compute $\rR ST(M)$. Since $S(T(I^i))=I^i$, claim (a) follows. The remaining claims follow easily.
\end{proof}

If $N$ is a finite length object of $\Mod_{\FI}^{\gen}$ then $N=T(M)$ for some finitely generated $\FI$-module $M$. Thus the above proposition gives finiteness results for $\rR^i S(N)$. We record one final result here.

\begin{proposition} \label{prop:FI-Ext-Pn}
Let $N$ be a finite length object of $\Mod_{\FI}^{\gen}$. The following are equivalent:
\begin{enumerate}
\item $S(N)$ is a semi-induced $\FI$-module.
\item $\Ext^i(\ol{\bP}_n, N)=0$ for all $n \ge 0$ and $i>0$.
\end{enumerate}
\end{proposition}

\begin{proof}
By derived adjunction (see \S \ref{ss:deradj}), we have a natural isomorphism
\begin{displaymath}
\rR \Hom(M, \rR S(N)) = \rR \Hom(T(M), N)
\end{displaymath}
where $M$ is an $\FI$-module and $N$ is an object of the generic category. Take $M=\bP_n$. The higher $\rR \Hom$'s on the left vanish since $M$ is projective. We thus find
\begin{displaymath}
(\rR^i S N)_n = \Ext^i(\ol{\bP}_n, N).
\end{displaymath}
Note that $N=TS(N)$. Thus the above shows that $\rR^i ST(SN)=0$ for all $i>0$, meaning $SN$ is derived saturated, if and only if (b) holds. The result thus follows from Proposition~\ref{prop:FI-dersat}.
\end{proof}

\subsection{Derived adjunction} \label{ss:deradj}

We will use a dervied version of adjunction several times, so we state it in the abstract here. Let $\cA$ and $\cB$ be abelian categories with enough injectives, let $F \colon \cA \to \cB$ be an exact functor, and let $G \colon \cB \to \cA$ be the right adjoint to $F$. Since $F$ is exact, $G$ carries injectives to injectives. Let $M \in \rD^+(\cA)$ and $N \in \rD^+(\cB)$. Then we have a natural isomorphism
\begin{displaymath}
\rR \Hom_{\cA}(M, \rR G(N)) = \rR \Hom_{\cB}(F(M), N),
\end{displaymath}
as one sees by computing with an injective resolution of $N$.

\section{The infinite symmetric group} \label{s:sinf}

In this section, we recall basic results about the representation theory of the infinite symmetric group $\fS$. A detailed treatment of this representation theory, in characteristic~0, is given in \cite{infrank}. We also review some concepts from $\fS$-equivariant commutative algebra.

\subsection{Smooth representations} \label{ss:smooth-rep}

Let $\fS=\bigcup_{n \ge 1} \fS_n$ denote the infinite symmetric group, acting on the set $[\infty]=\{1,2,\ldots\}$. For $n \ge 1$, we let $\fS(n)$ be the subgroup of $\fS$ consisting of elements that fix each of $1, \ldots, n$. We say that a representation of $\fS$ on a $k$-vector space $V$ is \defn{smooth} if for every $x \in V$ there is some $n$ such that $x$ is fixed by $\fS(n)$. We let $\Rep(\fS)$ be the category of smooth representations; throughout this paper, all representations of $\fS$ are smooth. One easily sees that the class of smooth representations is closed under subquotients, direct limits, and tensor products. We will see that it also has a generator (Remark~\ref{rmk:generator}), and so $\Rep(\fS)$ is a Grothendieck abelian category.

\begin{remark}
Since $\Rep(\fS)$ is a Grothendieck abelian category, it is complete. However, a limit of smooth representations, taken in the category of $k[\fS]$-modules, is typically not smooth. Limits in $\Rep(\fS)$ are computed by first taking the limit in the category of $k[\fS]$-modules, and then taking the maximal smooth subrepresentation. The latter operation is not exact, so one must be careful when working with limits. For example, products in $\Rep(\fS)$ are not exact, that is, Grothendieck's (AB4*) axiom does not hold. One can see this using reasoning similar to \cite[Proposition~3.1]{increp}.
\end{remark}

\subsection{The basic representations} \label{ss:Vrep}

Let $\VV=\bigoplus_{n \ge 1} ke_n$ be the usual permutation representation of $\fS$. This is smooth: indeed, a typical vector has the form $\sum_{i=1}^n a_i e_i$ for some $n$, and is fixed by $\fS(n)$. For $n \ge 0$, let $\VV_n$ be the subrepresentation of $\VV^{\otimes n}$ spanned by tensors
\begin{displaymath}
e_{i_1,\ldots,i_n} = e_{i_1} \otimes \cdots \otimes e_{i_n}
\end{displaymath}
with all indices distinct. Since $\VV^{\otimes n}$ is smooth and $\VV_n$ is a subrepresentation, it too is smooth. The $\VV_n$'s are the most important representations for us. They enjoy the following mapping property.

\begin{proposition} \label{prop:Vn}
Let $V$ be a smooth representation of $\fS$. We have a bijection
\begin{displaymath}
i \colon \Hom_{\fS}(\VV_n, V) \to V^{\fS(n)}, \qquad i(f) = f(e_{1,\ldots,n}).
\end{displaymath}
Moreover, $V$ is a quotient of a direct sum of $\VV_n$'s.
\end{proposition}

\begin{proof}
Since $\VV_n$ is generated by $e_{1,\ldots,n}$, it follows that $i$ is injective. Now let $x \in V^{\fS(n)}$ be given. Define a linear map $f \colon \VV_n \to V$ by $f(e_{i_1,\ldots,i_n})=\sigma x$, where $\sigma \in \fS$ is any element satisfying $\sigma(1)=i_1, \ldots, \sigma(n)=i_n$. We note that $\sigma$ is well-defined up to right multiplication by $\fS(n)$, and so $f$ is well-defined since $x$ is $\fS(n)$-invariant. One easily sees that $f$ is a map of representations. Since $i(f)=x$, we see that $i$ is surjective. Finally, let $\{x_i\}_{i \in I}$ be a family of generators for $V$. We have just seen that each $x_i$ is in the image of a map $\VV_{n_i} \to V$ for some $n_i$. We thus have a surjection $\bigoplus_{i \in I} \VV_{n_i} \to V$.
\end{proof}

\begin{remark} \label{rmk:generator}
It follows that $\bigoplus_{n \ge 0} \VV_n$ is a generating object of $\Rep(\fS)$.
\end{remark}

\subsection{The connection to \texorpdfstring{$\FI$}{FI}-modules} \label{ss:sym-FI}

Let $M$ be an $\FI$-module. Define
\begin{displaymath}
\Theta(M) = \varinjlim_{n \to \infty} M_n
\end{displaymath}
where we are taking the direct limit with respect to the standard inclusions $[n] \to [n+1]$. One easily sees that this carries a representation of $\fS$. A simple computation shows that we have an isomorphism of $\fS$-representations
\begin{displaymath}
\Theta(\bP_n) = \VV_n.
\end{displaymath}
If $M$ is an arbitrary $\FI$-module then, writing $M$ as a quotient of a sum of principal projectives, we see that $\Theta(M)$ is a quotient of a sum of $\VV_n$'s, and thus a smooth representation. Hence $\Theta$ defines a functor
\begin{displaymath}
\Theta \colon \Mod_{\FI} \to \Rep(\fS).
\end{displaymath}
It is clear from the definition that $\Theta$ kills torsion modules. Since $\Theta$ is exact, it therefore induces a functor
\begin{displaymath}
\ol{\Theta} \colon \Mod_{\FI}^{\gen} \to \Rep(\fS).
\end{displaymath}
We require the following important result.

\begin{proposition} \label{prop:FI-gen-sym}
The functor $\ol{\Theta}$ is an equivalence.
\end{proposition}

\begin{proof}
This result is well-known, but since we do not have a convenient reference we sketch a proof. Let $V$ be a representation of $\fS$. Then $V^{\fS(n)}$ is a representation of $\fS_n$, and there is a natural inclusion $V^{\fS(n)} \to V^{\fS(n+1)}$. One readily verifies that this defines the structure of an $\FI$-module, and so we have a functor
\begin{displaymath}
\Theta^* \colon \Rep(\fS) \to \Mod_{\FI}.
\end{displaymath}
If $M$ is an $\FI$-module then giving a map $M \to \Theta^*(V)$ of $\FI$-modules is easily seen to be equivalent to giving a map $\Theta(M) \to V$ of $\fS$-representations, and so $\Theta^*$ is right adjoint to $\Theta$. The smoothness condition shows that the natural map $V \to \Theta(\Theta^*(V))$ is an isomorphism. Finally, it follows from the definitions that $\Theta(M)=0$ if and only if $M$ is a torsion $\FI$-module. \cite[III Prop.~5]{Gabriel} thus shows that $\ol{\Theta}$ is an equivalence.
\end{proof}

\begin{corollary}
The representation $\VV_n$ has finite length, and the category $\Rep(\fS)$ is locally of finite length.
\end{corollary}

\begin{proof}
This follows from the corresponding properties of generic $\FI$-modules discussed in \S \ref{ss:FI-gen}. See also \cite[(6.1.6)]{infrank} for a proof in characteristic~0.
\end{proof}

\subsection{Induced modules}

For a representation $V$ of $\fS_n$, we put
\begin{displaymath}
\PP(V) = \Ind_{\fS_n \times \fS(n)}^{\fS}(V \boxtimes k).
\end{displaymath}
Here induction is defined in the usual manner, and is easily seen to preserve smooth representations. In fact, we have
\begin{displaymath}
\PP(V) = (\VV_n \otimes V)_{\fS_n},
\end{displaymath}
which also shows that $\PP(V)$ is smooth; as with $\FI$-modules, we can use invariants here instead of co-invariants (see Proposition~\ref{prop:ind-inv}). We call representations of the form $\PP(V)$ \defn{induced}. Since $\Theta$ is exact and maps $\bP_n$ to $\VV_n$, we find
\begin{displaymath}
\Theta(\bP(V)) = \PP(V).
\end{displaymath}
We require the following result about tensor products of induced modules.

\begin{proposition} \label{prop:ten-ind}
Let $V$ be a $k[\fS_n]$-module and let $W$ be a $k[\fS_m]$-module. We have a decomposition of $\fS$-modules
\begin{displaymath}
\PP(V) \otimes \PP(W) \cong \bigoplus_{r=0}^{\min(n,m)} \PP(U_r), \qquad
U_r = \Ind_{\fS_r \times \fS_{n-r} \times \fS_{m-r}}^{\fS_{n+m-r}}(V \otimes W).
\end{displaymath}
Here $\fS_r \times \fS_{n-r} \times \fS_{m-r}$ is the usual Young subgroup of $\fS_{n+m-r}$, and $\fS_r \times \fS_{n-r}$ acts on $V$, and $\fS_r \times \fS_{m-r}$ acts on $W$.
\end{proposition}

\begin{proof}
First consider the tensor product $\VV_n \otimes \VV_m$. This is the subspace of $\VV^{\otimes (n+m)}$ spanned by tensors
\begin{displaymath}
e_{i_1, \ldots, i_n, j_1, \ldots, j_m}
\end{displaymath}
where the $i$'s are distinct and the $j$'s are distinct. Now, the span of the vectors where the $i$'s and $j$'s are distinct in total is $\VV_{n+m}$. If $i_1=j_1$ and the other indices are distinct, we get a copy of $\VV_{n+m-1}$. Carrying on with this reasoning, we obtain a decomposition
\begin{displaymath}
\VV_n \otimes \VV_m \cong \bigoplus_{\Gamma} \VV_{n+m-\# E(\Gamma)}
\end{displaymath}
where the sum is over bipartite matchings $\Gamma$ on $A \amalg B$, and $E(\Gamma)$ is the edge set of $\Gamma$. Suppose $\Gamma$ has $r$ edges. Then the $\fS_n \times \fS_m$ orbit of $\Gamma$ contains the standard matching, whose edges are $(1,1), \ldots, (r,r)$. We thus see that, as a representation of $\fS_n \times \fS_m \times \fS$, we have
\begin{displaymath}
\VV_n \otimes \VV_m \cong \bigoplus_{r=0}^{\min(n,m)} \Ind_{\fS_r \times \fS_{n-r} \times \fS_{m-r}}^{\fS_n \times \fS_m}(\VV_{n+m-r}),
\end{displaymath}
where here $\fS_r$ is embedded diagonally in $\fS_n$ and $\fS_m$.

Now suppose that $V$ is a $k[\fS_n]$-module and $W$ is a $k[\fS_m]$-module. Tensor the above with $V \otimes W$ and take invariants under $\fS_n \times \fS_m$. By Frobenius reciprocity, we have
\begin{displaymath}
\PP(V) \otimes \PP(W) \cong \bigoplus_{r=0}^{\min(n,m)} (\VV_{n+m-r} \otimes V \otimes W)^{\fS_r \times \fS_{n-r} \times \fS_{m-r}}.
\end{displaymath}
Here $\fS_r \times \fS_{n-r} \times \fS_{m-r}$ is a Young subgroup of $\fS_{n+m-r}$, and its action on $\VV_{n+m-r}$ is through this inclusion. Another application of Frobenius reciprocity gives
\begin{displaymath}
\PP(V) \otimes \PP(W) \cong \bigoplus_{r=0}^{\min(n,m)} (\VV_{n+m-r} \otimes \Ind_{\fS_r \times \fS_{n-r} \times \fS_{m-r}}^{\fS_{n+m-r}}(V \otimes W))^{\fS_{n+m-r}},
\end{displaymath}
which gives the stated result.
\end{proof}

\subsection{Semi-induced modules}

As with $\FI$-modules, we say that a representation of $\fS$ is \defn{semi-induced} if it admits a finite length filtration where the graded pieces are induced. Since $\Theta$ is exact and maps induced $\FI$-modules to induced $\fS$-modules, it also maps semi-induced $\FI$-modules to semi-induced $\fS$-modules.

\begin{proposition} \label{prop:Ext-semi-induced-sinf}
Let $V$ be a finite length representation of $\fS$. Then $V$ is semi-induced if and only if
\begin{displaymath}
\Ext^i_{\fS}(\VV_n, V)=0
\end{displaymath}
holds for all $n \ge 0$ and $i>0$.
\end{proposition}

\begin{proof}
This follows from Proposition~\ref{prop:FI-Ext-Pn}.
\end{proof}

\begin{proposition} \label{prop:sym-semi-ind-res}
Let $V$ be a finite length representation of $\fS$. Then there is a finite length resolution
\begin{displaymath}
0 \to V \to I^0 \to \cdots \to I^n \to 0
\end{displaymath}
where each $I^i$ is of finite length and semi-induced.
\end{proposition}

\begin{proof}
This follows from Corollary~\ref{cor:semi-ind-res} by applying $\Theta$. In characteristic~0, it also follows from the results of \cite{infrank}.
\end{proof}

\subsection{Commutative algebra} \label{ss:commalg}

We now briefly review some general concepts about commutative algebra in the category $\Rep(\fS)$. An \defn{$\fS$-algebra} is a commutative $k$-algebra $A$ equipped with an action of $\fS$ by algebra automorphisms, under which it forms a smoth representation. In other words, an $\fS$-algebra is a commutative algebra object in the category $\Rep(\fS)$.

Let $A$ be an $\fS$-algebra. We write $\vert A \vert$ for the $k$-algebra underlying $A$; we use this notation when we want to forget the $\fS$ action. An \defn{$A$-module} is a module object in $\Rep(\fS)$. Concretely, an $A$-module is an $\vert A \vert$-module $M$, equipped with a smooth action of $\fS$ that is compatible with the one on $A$, meaning $\sigma(am)=\sigma(a) \sigma(m)$ for $a \in A$, $m \in M$, and $\sigma \in \fS$. We let $\Mod_A$ be the category of $A$-modules. It is a Grothendieck abelian category.

An $A$-module $M$ is \defn{finitely generated} if $\vert M \vert$ is generated as an $\vert A \vert$-module by finitely many $\fS$-orbits. (Here we use $\vert M \vert$ in an analogous manner to $\vert A \vert$.) Equivalently, $M$ is finitely generated if and only if it is a quotient of an $A$-module of the form $A \otimes V$ with $V$ a finite length representation. We say that $A$ is \defn{noetherian} if any submodule of a finitely generated $A$-module is finitely generated.

An \defn{$\fS$-ideal} of $A$ is an ideal of the ring $\vert A \vert$ that is stable for the $\fS$ action. An $\fS$-ideal $\fp$ is \defn{$\fS$-prime} if it is not the unit ideal, and $\fa \fb \subset \fp$ implies $\fa \subset \fp$ or $\fb \subset \fp$ for all $\fS$-ideals $\fa$ and $\fb$. Equivalently, $\fp$ is $\fS$-prime if whenever $a$ and $b$ are elements of $A$ such that $a \cdot \sigma(b)$ belongs to $\fp$ for all $\sigma \in \fS$, either $a \in \fp$ or $b \in \fp$ \cite[Proposition~2.9]{sideals}. We say that $A$ is an \defn{$\fS$-domain} if the zero ideal is $\fS$-prime.

Suppose $A$ is an $\fS$-domain. We say that a module $M$ is \defn{torsion} if every finitely generated submodule has non-zero annihilator. Equivalently, this means that for every $x \in M$ there exists $a \in A$ non-zero such that $\sigma(a) \cdot x=0$ for all $\sigma \in \fS$. We let $\Mod_A^{\tors}$ be the full subcategory of $\Mod_A$ spanned by torsion modules. This is a localizing subcategory (see below). We define the \defn{generic category} $\Mod_A^{\gen}$ to be the Serre quotient $\Mod_A/\Mod_A^{\tors}$.

\begin{proposition}
Let $A$ be an $\fS$-domain. Then $\Mod_A^{\tors}$ is a localizing subcategory of $\Mod_A$.
\end{proposition}

\begin{proof}
It is clear that the class of torsion modules is closed under passing to subquotients and arbitrary direct sums. Now, suppose that
\begin{displaymath}
0 \to M_1 \to M_2 \to M_3 \to 0
\end{displaymath}
is a short exact sequence of $A$-modules, with $M_1$ and $M_3$ torsion. We must show that $M_2$ is torsion. Let $N$ be a finitely generated submodule of $M_2$, and let $\ol{N}$ be its image in $M_3$. Since $M_3$ is torsion and $\ol{N}$ is a finitely generated submodule, there is a non-zero $\fS$-ideal $\fa$ of $A$ such that $\fa \ol{N}=0$. Of course, we can assume $\fa$ is finitely generated. Since $\fa N$ is a finitely generated submodule of the torsion module $M_1$, there is a non-zero $\fS$-ideal $\fb$ such that $\fb (\fa N)=0$. Since $A$ is an $\fS$-domain, $\fa \fb$ is non-zero. We have thus shown that $M_2$ is torsion, as required.
\end{proof}

\section{The polynomial ring \texorpdfstring{$R$}{R}}

In this section, we introduce the polynomial $R$, the quotient rings $R/\fh_s$, and the basic modules over these rings.

\subsection{The polynomial ring} \label{ss:poly-ring}

Let $R$ be the polynomial ring $k[x_i]_{i \ge 1}$. We let $\fS$ act on $R$ by permuting variables. This action is smooth, as the monomial $x_1^{e_1} \cdots x_n^{e_n}$ is fixed by $\fS(n)$. Thus $R$ is an $\fS$-algebra, as defined in \S \ref{ss:commalg}. We also have an identification $R=\Sym(\VV)$, which gives a coordinate-free description of $R$ as an $\fS$-algebra. The most fundamental theorem about $R$ is that it is noetherian \cite{Cohen, Cohen2, NagelRomer2}.

\subsection{The \texorpdfstring{$\fh$}{h} ideals}

For $s \ge 0$, let $\fh_s$ denote the $\fS$-ideal of $R$ generated by $x_i^{s+1}$, for $i \ge 1$. The following proposition is contained among the results of \cite{sideals}, but we include a proof since it is central to the present paper.

\begin{proposition} \label{prop:h-prime}
The ideal $\fh_s$ is $\fS$-prime. Moreover, the only $\fS$-primes containing $\fh_s$ are $\fh_0, \ldots, \fh_s$.
\end{proposition}

\begin{proof}
Suppose $f \cdot \sigma(g)$ belongs to $\fh_s$ for all $\sigma \in \fS$, where $f \in R \setminus \fh_s$ and $g \in R$. We must show $g \in \fh_s$. Choose $\sigma$ such that $f$ and $\sigma(g)$ contain no common variable. Write $f=\sum_{i=1}^n c_i m_i$, where the $c_i$'s are non-zero scalars and the $m_i$'s are distinct monomials. Since $\fh_s$ is a monomial ideal, we have $m_i \sigma(g) \in \fh_s$ for all $i$. Since $f \not\in \fh_s$, there is some $i$ such that $m_i \not\in \fh_s$. Since $m_i \sigma(g) \in \fh_s$ and $m_i$ and $\sigma(g)$ have no common variable, it follows that $\sigma(g) \in \fh_s$, and so $g \in \fh_s$, as required.

Now suppose $\fp$ is an $\fS$-prime containing $\fh_s$. Thus $x_1^{s+1} \in \fp$. Let $r \ge 0$ be minimal such that $x_1^{r+1} \in \fp$. We claim that $\fp=\fh_r$. Clearly, we have the containment $\fh_r \subset \fp$. We therefore have equality if $r=0$. Suppose then that $r \ge 1$ and the containment is strict. Then $\fp$ contains some monomial $x_1^{e_1} \cdots x_n^{e_n}$ where all exponents are $\le r$, and thus contains $(x_1 \cdots x_n)^r$. Let $n$ be minimal such that $(x_1 \cdots x_n)^r \in \fp$. We must have $n \ge 2$ by the minimality of $r$. We have $(x_1 \cdots x_{n-1})^r \cdot \sigma(x_n^r) \in \fp$ for all $\sigma \in \fS$. Indeed, if $\sigma(n) \not\in [n-1]$ then this product is simply in the $\fS$-orbit of $(x_1 \cdots x_n)^r$, while if $\sigma(n) \in [n-1]$ then some exponent is $>r$. Thus, since $\fp$ is $\fS$-prime, we have $(x_1 \cdots x_{n-1})^r \in \fp$, contradicting the minimality of $n$. This contradiction proves $\fp=\fh_r$, as required.
\end{proof}

The next proposition follows from general theory. Since it too is important to the present paper, we sketch the proof.

\begin{proposition} \label{prop:hs-tors-ann}
A finitely generated torsion $R/\fh_s$-module is killed by a power of $\fh_{s-1}$.
\end{proposition}

\begin{proof}
Let $M$ be a finitely generated torsion $R/\fh_s$-module. The annihilator $\fa$ of $M$ in $R$ is an $\fS$-ideal that properly contains $\fh_s$. The \defn{$\fS$-radical} of $\fa$, denoted $\rad_{\fS}(\fa)$, is the sum of all $\fS$-ideals $\fb$ such that $\fb^n \subset \fa$ for some $n$. One easily sees that it is an $\fS$-ideal. Furthermore, we have
\begin{displaymath}
\rad_{\fS}(\fa) = \bigcap_{\fa \subset \fq} \fq,
\end{displaymath}
where the intersection is over all $\fS$-primes $\fq$ containing $\fa$; see \cite[Proposition~2.23]{sideals}. Since $\fa$ properly contains $\fh_s$, it follows from Proposition~\ref{prop:h-prime} that $\rad_{\fS}(\fa)=\fh_r$ for some $0 \le r < s$. We thus have $\fh_{s-1} \subset \fa$. Since $R$ is noetherian, one easily sees that $\fa$ contains a power of its $\fS$-radical. Thus $\fh_{s-1}^n \subset \fa$ for some $n$, which completes the proof.
\end{proof}

The finite truncations of $R/\fh_s$ have some special ring-theoretic properties that will play an important role in our work.

\begin{proposition} \label{prop:S-prop}
Fix $n \ge 1$, and let $S=k[x_1, \ldots, x_n](x_i^{s+1})$. Then:
\begin{enumerate}
\item The $k$-linear dual $S^*$ of $S$ is a free $S$-module of rank~1.
\item $S$ is injective as an $S$-module.
\end{enumerate}
\end{proposition}

\begin{proof}
(a) For $\alpha \in \bN^n$, let $x^{\alpha}=x_1^{\alpha_1} \cdots x_n^{\alpha_n}$. Then the $x^{\alpha}$ with $0 \le \alpha_i \le s$ form a $k$-basis of $S$. Let $\{ y_{\alpha} \}$ be the dual basis of $S^*$. A simple computation shows that $x^{\alpha} y_{\beta}=y_{\beta-\alpha}$, which is taken to be~0 if any coordinate is negative. We thus see that $S^*$ is free with basis $y_{(s,\ldots,s)}$.

(b) Since $S$ is a finite dimension $k$-algebra, the $k$-linear dual of any finite dimensional projective $S$-module is an injective $S$-module. Thus $S \cong S^*$ is injective.
\end{proof}

\subsection{The \texorpdfstring{$\cP$}{P}-modules} \label{ss:Pmod}

Let $\cP_{s,n} = R/\fh_s \otimes \VV_n$. This is the free $R/\fh_s$-module with basis $e_{i_1,\ldots,i_n}$, where the $i$'s are distinct. Every $R/\fh_s$-module is a quotient of a direct sum of $\cP_{s,n}$'s. More generally, if $V$ is an $\fS_n$-module, we put
\begin{displaymath}
\cP_s(V) = R/\fh_s \otimes \PP(V) = (\cP_{s,n} \otimes V)_{\fS_n}
\end{displaymath}
In characteristic~0, we let $\cP_{s,\lambda}=\cP_s(S^{\lambda})$, which is the $S^{\lambda}$-multiplicity space of $\cP_{s,n}$ (since $S^{\lambda}$ is self-dual).

\subsection{The \texorpdfstring{$\cQ$}{Q}-modules} \label{ss:Qmod}

Define $\cQ_{s,n}$ to be the quotient of $\cP_{s,n}$ by the submodule generated by elements $x_i e_{i_1, \ldots, i_n}$, where $i$ is one of the indices $i_1, \ldots, i_n$. These modules, and some variants defined below, will play a crucial role in our analysis of $R/\fh_s$. We begin by recording the mapping property for the $\cQ$'s.

\begin{proposition} \label{prop:Qmapprop}
Let $M$ be a $R/\fh_s$-module, and let $M'$ be the set of elements $M$ that are fixed by $\fS(n)$ and annihilated by each $x_i$ for $1 \le i \le n$. Then we have a bijection
\begin{displaymath}
\Hom_{R/\fh_s}(\cQ_{s,n}, M) \to M', \qquad f \mapsto f(e_{1,\ldots,n}).
\end{displaymath}
\end{proposition}

\begin{proof}
This follows from the mapping property for $\VV_n$ (Proposition~\ref{prop:Vn}).
\end{proof}

We note some corollaries.

\begin{corollary} \label{cor:QQmaps}
The space $\Hom_{R/\fh_s}(\cQ_{s,n}, \cQ_{s,m})$ has a $k$-basis indexed by injections $[n] \to [m]$.
\end{corollary}

\begin{proof}
An $\fS(n)$-invariant element of $\cQ_{s,m}$ has the form
\begin{displaymath}
\sum_{i \colon [n] \to [m]} a(i) e_{i(1), \ldots, i(n)}
\end{displaymath}
where $a(i)$ is a polynomial in $x_1, \ldots, x_n$. For $i$ as above, let $b(i)$ be the product of the $x_j^s$ over $j \not\in \im(i)$. Then the above element is annihilated by $x_1, \ldots, x_n$ if and only if $a(i)$ is divisible by $b(i)$ for each $i$. Of course, we can assume that the variables $x_{i(1)}, \ldots, x_{i(n)}$ do not occur in $a(i)$. If $a(i)$ is divisible by $b(i)$ and also satisfies this condition then it is simply a $k$-linear multiple of $b(i)$. We thus see that the elements $b(i) e_{i(1), \ldots, i(n)}$ form a $k$-basis of $\cQ_{s,m}'$, to use the notation of Proposition~\ref{prop:Qmapprop}. The result follows.
\end{proof}

\begin{corollary}
We have an injective map of $R/\fh_s$-modules
\begin{displaymath}
f \colon \cQ_{s,n} \to \cP_{s,n}, \qquad f(e_{1,\ldots,n}) = (x_1 \cdots x_n)^s e_{1,\ldots,n}.
\end{displaymath}
Thus $\cQ_{s,n}$ can be realized as a submodule of $\cP_{s,n}$.
\end{corollary}

\begin{proof}
The mapping property for $\cQ_{s,n}$ ensures the existence of the map $f$. One easily sees that $f$ maps a $k$-basis of $\cQ_{s,n}$ to $k$-linearly independent elements, and so $f$ is injective.
\end{proof}

In fact, the above map is, in a sense, universal for maps from $Q$'s to $P$'s, as the following result shows.

\begin{proposition} \label{prop:QPmaps}
Let $n,m \in \bN$. Then any map $f \colon \cQ_{s,m} \to \cP_{s,n}$ has image contained in $\cQ_{s,n} \subset \cP_{s,n}$. In other words, the inclusion $\cQ_{s,n} \to \cP_{s,n}$ induces an isomorphism
\begin{displaymath}
\Hom_R(\cQ_{s,m}, \cQ_{s,n}) \to \Hom_R(\cQ_{s,m}, \cP_{s,n}).
\end{displaymath}
\end{proposition}

\begin{proof}
Let $f$ be given, and put $z=f(e_{1,\ldots,m})$. Since $z$ is an $\fS(m)$-invariant element of $\cP_{s,n}$, we have an expression
\begin{displaymath}
z = \sum_{i \colon [n] \to [m]} a(i) e_{i(1),\ldots,i(n)},
\end{displaymath}
where each $a_{i_1,\ldots,i_n}$ is a polynomial in $x_1, \ldots, x_m$. Since $z$ is killed by $x_1, \ldots, x_m$, it follows that each coefficient $a(i)$ is divisible by $(x_1 \cdots x_m)^s$, and thus simply a scalar multiple of this monomial. Thus $z$ is contained in $\cQ_{s,n}$. Since $e_{1,\ldots,m}$ generates $\cQ_{s,m}$, it follows that the image of $f$ is contained in $\cQ_{s,n}$.
\end{proof}

For an $\fS_n$-module $V$, we put
\begin{displaymath}
\cQ_s(V) = (\cQ_{s,n} \otimes V)^{\fS_n}.
\end{displaymath}
As usual, in characteristic~0 we put $\cQ_{s,\lambda}=\cQ_s(S^{\lambda})$; in this case, we have a decomposition
\begin{displaymath}
\cQ_{s,n} = \bigoplus_{\lambda \vdash n} S^{\lambda} \otimes \cQ_{s,\lambda}.
\end{displaymath}
Collectively, we refer to the modules $\cQ_s(V)$ as \defn{$\cQ_s$-modules}. Note that if $\chr(k)=0$ then any $\cQ_s$-module is a direct sum of $\cQ_{s,\lambda}$'s. We say that an $R/\fh_s$-module $M$ is \defn{$\cQ_s$-filtered} if there is a finite length filtration $0=F_0 \subset \cdots \subset F_n=M$ where each $F_i/F_{i-1}$ is a $\cQ_s$-module.

\subsection{A more general construction} \label{ss:indmod}

Fix $n \ge 0$. Let $S=k[x_i]/(x_i^{s+1})$ where $1 \le i \le n$, and let $S'=k[x_i]/(x_i^{s+1})$ where $n<i$, so that $R/\fh_s=S \otimes S'$. Let $\cC$ be the category of $(\fS_n \times \fS(n))$-equivariant $S$-modules. Then we have a restriction functor
\begin{displaymath}
\sF \colon \Mod_{R/\fh_s} \to \cC.
\end{displaymath}
This functor has a left adjoint $\sG$, given by
\begin{displaymath}
\sG(M) = \Ind_{\fS_n \times \fS(n)}^{\fS}(M \otimes S').
\end{displaymath}
The $\sG$ construction interpolates between the $\cP$ and $\cQ$ construction. Indeed, if $V$ is a representation of $\fS_n$ then
\begin{displaymath}
\cP_s(V)=\sG(S \otimes V), \qquad
\cQ_s(V)=\sG(S/\fm \otimes V),
\end{displaymath}
where $\fm$ is the homogeneous maximal ideal of $S$. In the above equations, $\fS(n)$ acts trivially on the input to $\sG$. The $\sG$ construction allows us to easily prove the following useful results.

\begin{proposition} \label{prop:PQfilt2}
The module $\cP_s(V)$ is $\cQ_s$-filtered.
\end{proposition}

\begin{proof}
The $\fm$-adic filtration on $S$ gives a filtration on $\cP_s(V)$ where the graded pieces are $\cQ_s(\fm^i/\fm^{i+1} \otimes V)$. Here we have used that the functor $\sG$ is exact.
\end{proof}

In the case of $\cP_{s,n}$, we can be more precise:

\begin{proposition} \label{prop:PQfilt}
The module $\cP_{s,n}$ carries a filtration $0=F_0 \subset \cdots \subset F_{(s+1)^n}=\cP_{s,n}$ where $F_i/F_{i-1}$ is isomorphic to $\cQ_{s,n}$ for each $i$.
\end{proposition}

\begin{proof}
We have just seen that $\cP_{s,n}$ carries a filtration with graded pieces $\cQ_s(\fm^i/\fm^{i+1} \otimes k[\fS_n])$. As a $k[\fS_n]$-module, $\fm^i/\fm^{i+1} \otimes[\fS_n]$ is isomorphic to a direct sum of $\dim_k(\fm^i/\fm^{i+1})$ copies of $k[\fS_n]$. Thus we can refine the filtration so that each graded piece is $\cQ_s(k[\fS_n])=\cQ_{s,n}$. The length of the refined filtration is $\dim_k(S)=(s+1)^n$.
\end{proof}

\begin{proposition} \label{prop:QP-res}
Let $V$ be a finite dimensional $k[\fS_n]$-module. We have a resolution
\begin{displaymath}
0 \to \cQ_s(V) \to \cP_{s,n}^{\oplus r_0} \to \cP_{s,n}^{\oplus r_1} \to \cdots
\end{displaymath}
for some $r_0, r_1, \ldots \in \bN$.
\end{proposition}

\begin{proof}
The $k$-linear dual of $S$ is isomorphic to $S$ as an $S$-module (Proposition~\ref{prop:S-prop}). Thus if we take the $k$-linear dual of the minimal free resolution of $S/\fm$, we obtain a resolution of $S$-modules
\begin{displaymath}
0 \to S/\fm \to S^{\oplus a_0} \to S^{\oplus a_1} \to \cdots.
\end{displaymath}
Similarly, dualizing a free resolution of $V^*$ as a $k[\fS_n]$-module, we obtain a resolution of $k[\fS_n]$-modules
\begin{displaymath}
0 \to V \to k[\fS_n]^{\oplus b_0} \to k[\fS_n]^{\oplus b_1} \to \cdots.
\end{displaymath}
Tensoring these resolutions together and taking the total complex, we obtain a resolution of $\fS_n$-equivariant $S$-modules
\begin{displaymath}
0 \to S/\fm \otimes V \to (S \otimes k[\fS_n])^{\oplus r_0} \to (S \otimes k[\fS_n])^{\oplus r_1} \to \cdots.
\end{displaymath}
We regard this as a sequence in $\cC$ by letting $\fS(n)$ act trivially. Applying $\sG$ yields the result.
\end{proof}

\subsection{A Tor calculation} \label{ss:gentor}

In classical commutative algebra, higher Tor groups are torsion modules. We now see that this is not the case in the present situation.

\begin{proposition}
For all $s \ge 1$ and $r \ge 1$ we have
\begin{displaymath}
\Tor^{R/\fh_s}_r(\cQ_{s,1}, \cQ_{s,1}) = \cQ_{s,1}.
\end{displaymath}
\end{proposition}

\begin{proof}
Use notation as in \S \ref{ss:indmod}, with $n=1$. Thus $S=k[x_1]/(x_1^{s+1})$. We have a 2-periodic free resolution
\begin{displaymath}
\xymatrix{
\cdots \ar[r] & S \ar[r]^-{x_1^s} & S \ar[r]^-{x_1} & S \ar[r]^-{x_1^s} & S \ar[r] & S/\fm \ar[r] & 0. }
\end{displaymath}
Applying the $\sG$ functor, we obtain a resolution
\begin{displaymath}
\xymatrix{
\cdots \ar[r] & \cP_{s,1} \ar[r] & \cP_{s,1} \ar[r] & \cP_{s,1} \ar[r] & \cP_{s,1} \ar[r] & \cQ_{s,1} \ar[r] & 0, }
\end{displaymath}
where the differentials map $e_i$ to $x_i^s e_i$ or $x_i e_i$. We can use this resolution to compute $\Tor^{R/\fh_s}$ since the module $\cP_{s,1}$ is flat over $R/\fh_s$. Applying $- \otimes_{R/\fh_s} \cQ_{s,1}$, we obtain the complex
\begin{displaymath}
\xymatrix{
\cdots \ar[r] & \VV \otimes \cQ_{s,1} \ar[r] & \VV \otimes \cQ_{s,1} \ar[r] & \VV \otimes \cQ_{s,1} \ar[r] & \VV \otimes \cQ_{s,1} }
\end{displaymath}
where each the differentials map $e_i \otimes m$ to $e_i \otimes x_i^s m$ or $e_i \otimes x_i m$. Ignoring the $\fS$ action, this complex decomposes as $\bigoplus_{i \ge 1} e_i \otimes C_i$, where $C_i$ is the complex
\begin{displaymath}
\xymatrix{
\cdots \ar[r] & \cQ_{s,1} \ar[r]^-{x_i^s} & \cQ_{s,1} \ar[r]^-{x_i} & \cQ_{s,1} \ar[r]^-{x_i^s} & \cQ_{s,1}. }
\end{displaymath}
Let $\cQ_{s,1}^{(i)}$ be the $|R|$-submodule of $\cQ_{s,1}$ generated by $e_i$. Then
\begin{displaymath}
\cQ_{s,1}[x_i^s] = x_i \cQ_{s,1} \oplus \cQ_{s,1}^{(i)}, \qquad
\cQ_{s,1}[x_i] = x_i^s \cQ_{s,1} \oplus \cQ_{s,1}^{(i)}.
\end{displaymath}
Here $(-)[a]$ denotes the set of elements killed by $a$. Note that this is indeed a direct sum. We thus see that $\rH_r(C_i)=\cQ_{1,1}^{(i)}$ for all $r \ge 1$. Hence
\begin{displaymath}
\Tor_r^{R/\fh_s}(\cQ_{s,1}, \cQ_{s,1}) = \bigoplus_{i \ge 1} e_i \otimes \rH_m(C_i) \cong \cQ_{s,1}
\end{displaymath}
as claimed.
\end{proof}

\begin{remark} \label{rmk:Q11-ext}
Using a similar resolution, and the fact that $\Ext^i_{R/\fh_s}(\cQ_{s,1}, \cP_{s,1})=0$ for $i\ge 1$ (Theorem~\ref{thm:vanish1}), one can show
\begin{displaymath}
\Ext^i_{R/\fh_s}(\cQ_{s,1,}, \cQ_{s,1}) = k
\end{displaymath}
for all $s \ge 1$ and $i \ge 0$.
\end{remark}

\section{From \texorpdfstring{$\FI$}{FI}-modules to \texorpdfstring{$R$}{R}-modules} \label{s:Phi}

\subsection{The main result}

Fix $s \ge 1$. In this section, we construct a functor
\begin{displaymath}
\Phi_s \colon \Mod_{\FI} \to \Mod_{R/\fh_s},
\end{displaymath}
and prove the following theorem about it.

\begin{theorem} \label{thm:Phi}
We have the following.
\begin{enumerate}
\item The functor $\Phi_s$ is exact, faithful, and commutes with arbitrary colimits.
\item If $M$ and $N$ are finitely generated $\FI$-modules, and $N$ is saturated, then the map
\begin{displaymath}
\Phi_s \colon \Hom_{\FI}(M,N) \to \Hom_R(\Phi_s(M), \Phi_s(N))
\end{displaymath}
is an isomorphism.
\item We have $\Phi_s(\bP_n) = \cQ_{s,n}$ and $\Phi_s(\bT_n) = \cQ_{s-1,n}$. More generally, for an $\fS_n$-representation $V$, we have $\Phi_s(\bP(V))=\cQ_s(V)$ and $\Phi_s(\bT(V))=\cQ_{s-1}(V)$.
\end{enumerate}
\end{theorem}

This functor $\Phi_s$ is quite useful, as it will allow us to transfer results about $\FI$-modules to $R/\fh_s$-modules.

\begin{remark} \label{rmk:Psi-not-full}
The functor $\Phi_s$ is not full. Indeed, if $n<m$ then $\Hom_{\FI}(\bT_m, \bT_n)=0$ but $\Hom_{R/\fh_s}(\cQ_{s-1,m}, \cQ_{s-1,n}) \ne 0$ (Corollary~\ref{cor:QQmaps}).
\end{remark}

\subsection{A preliminary construction}

A \defn{weighted set} is a finite set $S$ equipped with a function $\vert \cdot \vert \colon S \to [\infty]$. A \defn{morphism} of weighted sets $f \colon S \to T$ is an injective function such that $\vert x \vert \ge \vert f(x) \vert$, i.e., weights can not decrease. We let $\WI$ be the category of weighted sets. A \defn{$\WI$-module} is a functor from $\WI$ to the category of vector spaces. We let $\Mod_{\WI}$ be the cateory of $\WI$-modules.

Let $\Delta$ be the set of all functions $\alpha \colon [\infty] \to \bN$ with finite support. Let $V$ be a representation of $\fS$, and suppose that $V=\bigoplus_{\alpha \in \Delta} V_{\alpha}$ is a $\Delta$-grading of $V$. We say that the action and grading are \defn{compatible} if (a) $\sigma(V_{\alpha})=V_{\sigma \alpha}$ for $\sigma \in \fS$; and (b) if $\sigma$ acts by the identity on the support of $\alpha$ then $\sigma$ acts by the identity on $V_{\alpha}$. Note that (b) implies that the action of $\fS$ on $V$ is smooth.

For $\alpha \in \Delta$, let $x^{\alpha}$ denote the monomial $\prod_{i \ge 1} x^{\alpha_i}$. These monomials form a $k$-basis of $R$, and so endow $R$ with a $\Delta$-grading, which is easily seen to be compatible with the $\fS$-action. A \defn{$\Delta$-graded $R$-module} is an $R$-module $M$ equipped with a $\Delta$-grading that is compatible with the $\fS$-action and the $R$-grading, meaning that $x^{\alpha} M_{\beta} \subset M_{\alpha+\beta}$. We let $\Mod^{\Delta}_R$ denote the category of $\Delta$-graded $R$-modules.

For $\alpha \in \Delta$, we let $[\alpha]$ denote the support of $\alpha$, regarded as a weighted set with weight function $\alpha$. Let $M$ be a $\WI$-module. Define a vector space
\begin{displaymath}
N = \bigoplus_{\alpha \in \Delta} N_{\alpha}, \qquad N_{\alpha}=M([\alpha]).
\end{displaymath}
We now give $N$ the structure of a $\Delta$-graded $R$-module. The $\Delta$-grading is clear. Suppose that $\sigma \in \fS$ and $m \in N_{\alpha}$. Then $\sigma$ defines a morphism $i \colon [\alpha] \to [\sigma \alpha]$ in $\WI$, and we define $\sigma m \in N_{\sigma \alpha}$ to be $i_*(m)$. One easily sees that this defines a representation of $\fS$ on the space $N$ that is compatible with the grading. Next suppose $\alpha,\beta \in \Delta$ and $m \in N_{\beta}$. We have a natural morphism $j \colon [\alpha] \to [\alpha+\beta]$ in $\WI$, and we define $x^{\beta} m$ to be $j_*(m)$. One readily verifies that this defines an action of $R$ satisfying the necessary conditions.

The above construction defines a functor
\begin{displaymath}
\Phi' \colon \Mod_{\WI} \to \Mod^{\Delta}_R.
\end{displaymath}
In fact, one can show that this functor is an equivalence. We omit the proof, as we do not need it.

\subsection{A bounded variant}

An \defn{$s$-weighted set} is a weighted sets where the weights are at most $s$. Let $\WI(s)$ be the full subcategory of $\WI$ spanned by $s$-weighted sets. We let $\Mod_{\WI(s)}$ be the category of $\WI(s)$-modules. This category was introdued in \cite{FInmod}, with an eye towards the present application. Note that for $s=1$, the category $\WI(s)$ reverts to the familiar $\FI$. Also, if $M$ is a $\WI(s)$-module then we can extend $M$ by zero to a $\WI$-module, i.e., we define $M(S)=0$ if $S$ is a weighted set with some weight exceeding $s$.

Let $\Delta(s)$ be the subset of $\Delta$ consisting of functions valued in the set $\{0, \ldots, s\}$. A \defn{$\Delta(s)$-graded $R$-module} is a $\Delta$-graded $R$-module where the grading is supported on $\Delta(s)$. We let $\Mod^{\Delta(s)}_R$ denote the category of such modules. A basic example is $R/\fh_s$.

The functor $\Phi$ defined above induces a functor
\begin{displaymath}
\Phi'_s \colon \Mod_{\WI(s)} \to \Mod^{\Delta(s)}_R.
\end{displaymath}
Again, one can show that it is an equivalence.

\subsection{The main construction}

We have a functor
\begin{displaymath}
\Pi_s \colon \WI(s) \to \FI
\end{displaymath}
that takes an $s$-weighted set to its elements of weight $s$. This induces a pull-back functor on modules: that is, if $M$ is an $\FI$-module then $\Pi_s^*(M)$ is a $\WI(s)$-module. We define the functor $\Phi_s$ to be the composition
\begin{displaymath}
\xymatrix@C=3em{
\Mod_{\FI} \ar[r]^-{\Pi_s^*} \ar[r] &
\Mod_{\WI(s)} \ar[r]^-{\Phi'_s} &
\Mod^{\Delta(s)}_R \ar[r] & \Mod_R }
\end{displaymath}
where the final functor forgets the grading. It is clear that $\Phi_s$ is exact and commutes with colimits, as this is true for each individual functor above. We prove the remaining statements in Theorem~\ref{thm:Phi} in the subsequent lemmas.

\begin{lemma} \label{lem:Phi-1}
For an $\fS_n$-representation $V$, we have $\Phi_s(\bP(V))=\cQ_s(V)$.
\end{lemma}

\begin{proof}
We first suppose $V$ is the regular representation. Let $M=\Phi_s(\bP(V)) = \Phi_s(\bP_n)$. By definition, we have
\begin{displaymath}
M = \bigoplus_{\alpha \in \Delta(s)} \bP_n(\Pi_s([\alpha])).
\end{displaymath}
Recall that $\bP_n(S)=k[\Hom_{\FI}([n], S)]$. Let $\alpha \in \Delta(s)$ be defined by $\alpha_i=s$ for $i \in [n]$ and $\alpha_i=0$ for $i \not\in [n]$. Then $\Pi_s([\alpha])$ is the set of bijections $[n] \to [n]$, and has a canonical element (the identity). Let $m \in M$ be this canonical element of degree $\alpha$. Since $\alpha$ is supported on $[n]$, it follows that $m$ is fixed by $\fS(n)$. Also, if $\beta$ is a weight with $\beta_i>0$ for some $i \in [n]$ then $\alpha+\beta \not\in \Delta(s)$, and so $x^{\beta} m=0$. In particular, $m$ is killed by $x_1, \ldots, x_n$. It follows that there is a map of $R$-modules $\cQ_{s,n} \to M$ given by $e_{1,\ldots,n} \mapsto m$ (Proposition~\ref{prop:Qmapprop}). One easily sees that this induces a bijection on the natural bases of the source and target, and is thus an isomorphism.

We now treat the case of a general representation $V$. Choose a presentation $F_1 \to F_0 \to V \to 0$ where $F_0$ and $F_1$ are free $k[\fS_n]$-modules. Consider the diagram
\begin{displaymath}
\xymatrix{
\Phi_s(\bP(F_1)) \ar[r] \ar[d] & \Phi_s(\bP(F_0)) \ar[r] \ar[d] & \Phi_s(\bP(V)) \ar[r] \ar@{..>}[d] & 0 \\
\cQ_s(F_1) \ar[r] & \cQ_s(F_0) \ar[r] & \cQ_s(V) \ar[r] & 0 }
\end{displaymath}
where the left two vertical maps are the isomorphisms from the previous paragraph. One easily sees that the left square commutes. The rows are exact, since $\bP$, $\Phi_s$, and $\cQ_s$ are exact functors. It follows that the right map exists and is an isomorphism.
\end{proof}

\begin{lemma} \label{lem:Phi-2}
For an $\fS_n$-representation $V$, we have $\Phi_s(\bT(V))=\cQ_{s-1}(V)$.
\end{lemma}

\begin{proof}
First suppose $V$ is the regular representation. Let $M=\Phi_s(\bT(V))=\Phi_s(\bT_n)$. By definition, we have
\begin{displaymath}
M = \bigoplus_{\alpha \in \Delta(s)} \bT_n(\Pi_s([\alpha])).
\end{displaymath}
Recall that $\bT_n(S)=k[\operatorname{Bij}([n], S)]$. Let $\alpha \in \Delta(s)$ be as in the previous proof. Then once again, we have a canonical element $m \in M$ of degree $\alpha$, which is killed by $x_1, \ldots, x_n$. If $\beta \in \Delta$ has $\beta_i=s$ for some $i>n$ then $\Pi_s(\alpha+\beta)$ contains $[n] \cup \{i\}$, and thus is not in bijection with $[n]$; hence $M_{\alpha+\beta}=0$. It follows that $m$ is killed by $x_i^s$ for all $i>n$. Hence there is a map $\cQ_{s-1,n} \to M$ given by $e_{1,\ldots,n} \mapsto m$ (Proposition~\ref{prop:Qmapprop}). Again, one checks that it is a bijection on basis vectors. The case of general $V$ follows as in the proof of Lemma~\ref{lem:Phi-1}.
\end{proof}

\begin{lemma}
If $M$ and $N$ are finitely generated $\FI$-modules, and $N$ is saturated, then the map
\begin{displaymath}
\Phi_s \colon \Hom_{\FI}(M,N) \to \Hom_R(\Phi_s(M), \Phi_s(N))
\end{displaymath}
is an isomorphism.
\end{lemma}

\begin{proof}
We prove the lemma in five steps. All modules below are finitely generated.

\textit{Step 1: $M$ and $N$ are projective.} It suffices to treat the case $M=\bP_m$ and $N=\bP_n$. We have
\begin{displaymath}
\Hom_{\FI}(\bP_m, \bP_n)=\bP_n([m]) = k[\Hom_{\FI}([n], [m])].
\end{displaymath}
Similarly,
\begin{displaymath}
\Hom_R(\Phi_s(\bP_m), \Phi_s(\bP_n)) = \Hom_R(\cQ_{s,m}, \cQ_{s,n})
\end{displaymath}
also has a $k$-basis indexed by injections $[n] \to [m]$ (Corollary~\ref{cor:QQmaps}). Since the map in question is injective and the source and target have the same dimension, it is an isomorphism.

\textit{Step 2: $M$ is projective and $N$ is induced.} Say $N=\bP(V)$, where $V$ is a representation of $\fS_n$. Choosing a co-presentation of $V$ by free $k[\fS_n]$-modules and applying $\bP$, we obtain an exact sequence $0 \to N \to P_0 \to P_1$ with $P_0$ and $P_1$ projective. Consider the diagram
\begin{displaymath}
\xymatrix{
0 \ar[r] & \Hom(M, N) \ar[r] \ar[d] & \Hom(M, P_0) \ar[r] \ar[d] & \Hom(M, P_1) \ar[d] \\
0 \ar[r] & \Hom(\Phi_s(M), \Phi_s(N)) \ar[r] & \Hom(\Phi_s(M), \Phi_s(P_0)) \ar[r] & \Hom(\Phi_s(M), \Phi_s(P_1))
}
\end{displaymath}
The right two vertical maps are isomorphisms by Step~1. It follows that the left vertical map is an isomorphism as well.

\textit{Step 3: $M$ is projective and $N$ is semi-induced.} Choose an exact sequence
\begin{displaymath}
0 \to N' \to N \to N'' \to 0
\end{displaymath}
with $N'$ induced and $N''$ semi-induced with a shorter filtration length than $N$. Consider the diagram analogous to the one in Step~2. The left vertical map is an isomorphism by Step~2, and the right vertical map is an isomorphism by induction. Since the top row is exact, as $M$ is projective, it follows that the middle vertical map is an isomorphism.

\textit{Step 4: $M$ is projective and $N$ is saturated.} Let $0 \to N \to I^0 \to I^1$ be an exact sequence with $I^0$ and $I^1$ semi-induced; this exists since $N$ is saturated (Proposition~\ref{prop:FI-sec}). Now proceed as in Step~2, making use of Step~3.

\textit{Step 5: the general case.} Let $P_1 \to P_0 \to M\to 0$ be a presentation of $M$, where $P_0$ and $P_1$ are projective $\FI$-modules. Consider the diagram
\begin{displaymath}
\xymatrix{
0 \ar[r] & \Hom(M, N) \ar[r] \ar[d] & \Hom(P_0, N) \ar[r] \ar[d] & \Hom(P_1, N) \ar[d] \\
0 \ar[r] & \Hom(\Phi_s(M), \Phi_s(N)) \ar[r] & \Hom(\Phi_s(P_0), \Phi_s(N)) \ar[r] & \Hom(\Phi_s(P_1), \Phi_s(N))
}
\end{displaymath}
The right two vertical maps are isomorphisms by Step~4. It follows that the left vertical map is an isomorphism as well.
\end{proof}

\subsection{The adjoint} \label{ss:Phi-adjoint}

Since $\Phi_s$ is a cocontinuous functor of Grothendieck abelian categories, it admits a right adjoint
\begin{displaymath}
\Phi_s^* \colon \Mod_{R/\fh_s} \to \Mod_{\FI}.
\end{displaymath}
We now compute this functor in two cases.

\begin{proposition}
If $M$ is a saturated $\FI$-module then the unit $M \to \Phi_s^*(\Phi_s(M))$ is an isomorphism.
\end{proposition}

\begin{proof}
Let $N$ be an $\FI$-module. Then the map
\begin{displaymath}
\Hom_{\FI}(N, M) \to \Hom_{\FI}(N, \Phi^*(\Phi_s(M))) = \Hom_{R/\fh_s}(\Phi_s(N), \Phi_s(M))
\end{displaymath}
is simply the map induced by $\Phi_s$, which is an isomorphism by Theorem~\ref{thm:Phi}. The result thus follows from Yoneda's lemma.
\end{proof}

\begin{proposition} \label{prop:Phi-star}
We have $\Phi_s^*(\cP_{s,n})=\bP_n$, and, more generally, $\Phi_s^*(\cP_s(V)) = \bP(V)$.
\end{proposition}

\begin{proof}
Let $i \colon \cQ_{s,n} \to \cP_{s,n}$ be the usual inclusion. Let $M$ be an $\FI$-module, and suppose given a map $f \colon \Phi_s(M) \to \cP_{s,n}$. Choose a surjection $P \to M$ where $P$ is a sum of principal projective modules. We then have a surjection $\Phi_s(P) \to \Phi_s(M)$ where $\Phi_s(P)$ is a sum of $\cQ_{s,m}$'s. The restriction of $f$ to $\Phi_s(P_0)$ maps into $i(\cQ_{s,n})$ by Proposition~\ref{prop:QPmaps}, and so $f$ does as well, that is, $f$ factors uniquely as $i \circ g$ for some $g \colon \Phi_s(M) \to \cQ_{s,n}$. Now, $\cQ_{s,n}=\Phi_s(\bP_n)$. Since $\bP_n$ is saturated (Proposition~\ref{prop:FI-dersat}), it follows from Theorem~\ref{thm:Phi} that $g=\Phi_s(g')$ for a unique map $g' \colon M \to \bP_n$. We have thus shown that $f$ factors uniquely as $i \circ \Phi_s(g')$, and so $\Phi_s^*(P_n)=\bP_n$.

Now, let $V$ be a $k[\fS_n]$-module, so that $\cP_s(V)=(\cP_{s,n} \otimes V)^{\fS_n}$. We have
\begin{displaymath}
\Phi_s^*(\cP_s(V)) = (\Phi_s^*(\cP_{s,n}) \otimes V)^{\fS_n} = (\bP_n \otimes V)^{\fS_n} = \bP(V).
\end{displaymath}
In the first step, we used that $\Phi_s^*$ is left exact (as it is a right adjoint), and so commutes with $\fS_n$-invariants; we also used that $\Phi_s^*$ is additive, and so commutes with $-\otimes V$. In the second step, we used the previous paragraph. This completes the proof.
\end{proof}

\section{Some Ext vanishing} \label{s:ext}

The purpose of this section is to establish two important $\Ext$ vanishing theorems.

\subsection{The first theorem}

This gives $\Ext$ vanishing between certain $\cP$- and $\cQ$-modules:

\begin{theorem} \label{thm:vanish1}
We have $\Ext^i_{R/\fh_s}(\cQ_{s,n}, \cP_{s,d})=0$ for all $i>0$.
\end{theorem}

Before giving the proof, we note a few consequences.

\begin{corollary} \label{cor:vanish1}
We have $\rR^i \Phi_s^*(\cP_{s,n})=0$ for $i>0$.
\end{corollary}

\begin{proof}
By derived adjunction (\S \ref{ss:deradj}), for an $\FI$-module $M$ and an $R/\fh_s$-module $N$, we have
\begin{displaymath}
\rR \Hom_{\FI}(M, \rR \Phi_s^*(N)) = \rR \Hom_{R/\fh_s}(\Phi_s(M), N).
\end{displaymath}
We apply this with $M=\bP_n$. Since $\bP_n$ is projective, the higher $\rR \Hom$'s on the left vanish. We thus find
\begin{displaymath}
(\rR^i \Phi_s^*(N))_n = \Ext^i_{R/\fh_s}(\cQ_{s,n}, N).
\end{displaymath}
Taking $N=\cP_{s,n}$ yields the stated result.
\end{proof}

\begin{corollary}
Suppose $\chr(k)=0$. Let $M$ be a $\cQ_s$-filtered $R/\fh_s$-module and let $d \in \bN$. Then $\Ext^i_{R/\fh_s}(M, \cP_{s,d})=0$ hold for all $i>0$.
\end{corollary}

\begin{proof}
First note that the statement holds if $M=\cQ_{s,\lambda}$, as such a module is a summand of $\cQ_{s,n}$ with $n=\vert \lambda \vert$. The general case then follows from d\'evissage.
\end{proof}

\begin{remark}
If $M$ and $N$ are $\cQ_s$-filtered $R/\fh_s$-modules then $\Ext^i_{R/\fh_s}(M, N)$ can be non-zero. Indeed, not every $\cQ_s$-filtered module is a sum of $\cQ_s$-modules (e.g., $\cP_{s,d}$), and so there must be non-trivial $\Ext^1$ classes between $\cQ_s$-modules. Thus, in the above theorem, $\cP_{s,d}$ cannot be replaced with a general $\cQ_s$-filtered module.
\end{remark}

We now begin proving Theorem~\ref{thm:vanish1}. Fix $n \ge 0$ and let $S$, $S'$, $\cC$, $\sF$, and $\sG$ be as in \S \ref {ss:indmod}.

\begin{lemma} \label{lem:vanish1-1}
As an $\fS$-module, $R/\fh_s$ decomposes into a direct sum of finite length induced modules. (The sum is infinite if $s>0$.)
\end{lemma}

\begin{proof}
Let $X$ be the set of monomials in $R/\fh_s$. Then $R/\fh_s$ is identified with the permutation module $k[X]$. Let $X=\coprod_{i\in I} X_i$ be the orbit decomposition of $X_i$, so that $R/\fh_s$ decomposes as $\bigoplus_{i \in I} k[X_i]$. Fix $i \in I$. Then $X_i$ contains a unique monomial $m$ of the form $x_1^{e_1} \cdots x_n^{e_n}$ where $e_1 \ge \cdots \ge e_n \ge 1$. The stabilizer in $\fS$ of $m$ is $G \times \fS(n)$, where $G$ is a Young subgroup of $\fS_n$. We thus have
\begin{displaymath}
k[X_i] \cong \Ind_{G \times \fS(n)}^{\fS}(k) \cong \PP(k[\fS_n/G]),
\end{displaymath}
and so $k[X_i]$ is a finite length induced module.
\end{proof}

Recall that $\cC$ is the category of $(\fS_n \times \fS(n))$-equivariant modules over $S$. Thus if $W$ and $V$ are representations of $\fS_n$ and $\fS(n)$ then $S \otimes W \otimes V$ is an object of $\cC$. Note that $\fS(n)$ is isomorphic to $\fS$, so we can apply our teminology and results about $\fS$ to it.

\begin{lemma}
$\sF(\cP_{s,d})$ is a direct sum of objects of the form $S \otimes W \otimes V$, where $W$ is a finite dimensional $\fS_n$-representation and $V$ is a finite length induced $\fS(n)$-representation.
\end{lemma}

\begin{proof}
Recall that $\sF(\cP_{s,d})$ is simply $\cP_{s,d}$ regarded as an $\fS_n \times \fS(n)$ equivariant module over $S$. By definition, $\cP_{s,d} = R/\fh_s \otimes \VV_d$. We have $R/\fh_s = S \otimes S'$. Let $\VV_i(n)$ denote the subspace of $\VV_i$ where all indices are $\ge n$; this is identified with $\VV_i$ when $\fS(n)$ is identified with $\fS$. As a representation of $\fS_n \times \fS(n)$, we have
\begin{displaymath}
\VV_d=\bigoplus_{i=0}^d W_i \otimes \VV_i(n)
\end{displaymath}
for certain $\fS_n$-modules $W_i$. We thus have
\begin{displaymath}
\cP_{s,d} \cong \bigoplus_{i=0}^d S \otimes W_i \otimes (S' \otimes \VV_i(n)).
\end{displaymath}
Since $(S', \fS(n))$ is isomorphic to $(R/\fh_s, \fS)$, Lemma~\ref{lem:vanish1-1} shows that $S'$ is a direct sum of induced $\fS(n)$-modules. Since a tensor product of induced modules is a sum of induced modules (Proposition~\ref{prop:ten-ind}), the result follows.
\end{proof}

\begin{lemma}
Let $W$ and $V$ be as in the previous lemma. Then
\begin{displaymath}
\Ext^i_{\cC}(S/\fm \otimes k[\fS_n], S \otimes W \otimes V)=0
\end{displaymath}
for all $i>0$.
\end{lemma}

\begin{proof}
Let $\cC'$ be the category of $\fS_n$-equivariant modules over $S$. One easily sees that we have an isomorphism
\begin{displaymath}
\rR \Hom_{\cC}(S/\fm \otimes k[\fS_n], S \otimes W \otimes V) \cong \rR \Hom_{\cC'}(S/\fm \otimes k[\fS_n], S \otimes W) \otimes_k \rR \Hom_{\fS(n)}(k, V).
\end{displaymath}
The second factor is concentrated in degree~0 by Proposition~\ref{prop:Ext-semi-induced-sinf}. We have
\begin{displaymath}
\rR \Hom_{\cC'}(S/\fm \otimes k[\fS_n], S \otimes W) \cong
\rR \Hom_{\cC'}(S/\fm, S \otimes k[\fS_n] \otimes W)
\end{displaymath}
since $k[\fS_n]$ is a self-dual representation. Furthermore, $k[\fS_n] \otimes W$ is isomorphic to a direct sum of $\dim(W)$ copies of $k[\fS_n]$. The result thus follows from the fact that $S \otimes k[\fS_n]$ is injective in $\cC'$. Indeed, if $M$ is an object of $\cC'$ then
\begin{displaymath}
\Hom_{\cC'}(M, S \otimes k[\fS_n]) = \Hom_S(M, S),
\end{displaymath}
and $S$ is a self-injective (Proposition~\ref{prop:S-prop}), so the above functor is exact in $M$.
\end{proof}

\begin{proof}[Proof of Theorem~\ref{thm:vanish1}]
By derived adjunction (\S \ref{ss:deradj}), for $M \in \cC$ and $N \in \Mod_{R/\fh_s}$-module we have a natural identification
\begin{displaymath}
\Ext^i_{R/\fh_s}(\sG(M), N) = \Ext^i_{\cC}(M, \sF(N)).
\end{displaymath}
Note that $\sF$ is exact, so we do not need to derive it. We apply this with $M=S/\fm \otimes k[\fS_n]$ and $N=\cP_{s,d}$. We find
\begin{displaymath}
\Ext^i_{R/\fh_s}(\cQ_{s,n}, N) = \Ext^i_{\cC}(S/\fm \otimes k[\fS_n], \sF(\cP_{s,d})).
\end{displaymath}
This vanishes by the previous two lemmas. Note that since $\cC$ is locally noetherian, an arbitrary direct sum of injectives is injective, and so if $X$ is a finite length object then $\Ext^i_{\cC}(X, -)$ commutes with arbitrary direct sums (compute with injective resolutions in the second argument).
\end{proof}

\subsection{The second theorem}

The second theorem gives $\Ext$ vanishing between $\cQ_s$-filtered modules and torsion modules.

\begin{theorem} \label{thm:vanish2}
Let $M$ be a finitely generated $\cQ_s$-filtered $R/\fh_s$-module, and let $T$ be a torsion $R/\fh_s$-module. Then
\begin{displaymath}
\Ext^i_{R/\fh_s}(T, M)=0
\end{displaymath}
hold for all $i \ge 0$.
\end{theorem}

The following lemma is the key result.

\begin{lemma} \label{lem:vanish2-1}
For any $n,m \in \bN$, we have
\begin{displaymath}
\Ext^i_{R/\fh_s}(\cQ_{s-1,m},\cP_{s,n})=0, \qquad
\Ext^i_{R/\fh_s}(\cP_{s-1,m},\cP_{s,n})=0,
\end{displaymath}
for all $i \ge 0$.
\end{lemma}

\begin{proof}
By derived adjunction (\S \ref{ss:deradj}), we have
\begin{displaymath}
\rR \Hom_{\FI}(M, \rR \Phi_s^*(N)) = \rR \Hom_{R/\fh_s}(\Phi_s(M), N),
\end{displaymath}
for an $\FI$-module $M$ and an $R/\fh_s$-module $N$. We apply this with $N=\cP_{s,n}$. By Proposition~\ref{prop:Phi-star} and Corollary~\ref{cor:vanish1}, we have $\rR \Phi_s^*(\cP_{s,n}) = \bP_n$. We thus find
\begin{displaymath}
\Ext^i_{\FI}(M, \bP_n) = \Ext^i_{R/\fh_s}(\Phi_s(M), \cP_{s,n}).
\end{displaymath}
Now take $M=\bT_m$. The left side vanishes for all $i$ by Proposition~\ref{prop:FI-dersat}. Since $\Phi_s(\bT_m)=\cQ_{s-1,m}$, the first result follows. The second follows from the first, since $\cP_{s-1,m}$ has a finite length filtration with graded pieces $\cQ_{s-1,m}$ (Proposition~ \ref{prop:PQfilt}).
\end{proof}

\begin{proof}[Proof of Theorem~\ref{thm:vanish2}]
By d\'evissage, we reduce to the case $M=\cQ_s(V)$. If $T=\varinjlim T_i$ then
\begin{displaymath}
\rR \Hom_{R/\fh_s}(T, M) = \rR \varprojlim \rR \Hom_{R/\fh_s}(T_i, M),
\end{displaymath}
and so it suffices to treat the case where $T$ is finitely generated, which we assume. By another d\'evissage, we reduce to the case where $T$ is killed by $\fh_{s-1}$ (Proposition~\ref{prop:hs-tors-ann}). Since $T$ is an $R/\fh_{s-1}$-module, it admits a resolution $X_{\bullet} \to T$, where each $X_i$ is a sum of $\cP_{s-1,m}$'s. By Proposition~\ref{prop:QP-res}, we have a resolution $M \to Y_{\bullet}$, where each $Y_i$ is a finite sum of $\cP_{s,n}$'s. Thus the result follows from the previous lemma.
\end{proof}

\subsection{A variant}

We have just seen that $\Ext^i_{R/\fh_s}(T, M)$ vanishes in certain cases. We now observe that it is possible to prove the same result but working with $\Ext$ over other rings. Before proving this, we need a form of derived tensor-Hom adjunction.

\begin{proposition}
Let $A \to B$ be an $\fS$-algebra homomorphism, let $M$ be an $A$-module, and let $N$ be a $B$-module. Suppose that $A$ is noetherian and $M$ is finitely generated. Then we have a natural isomorphism
\begin{displaymath}
\rR \Hom_A(M, N) = \rR \Hom_B(B \otimes_A^{\rL} M, N).
\end{displaymath}
\end{proposition}

\begin{proof}
Typically, this kind of result is proved using a projective resolution of $M$ as an $A$-module. Since our categories do not have enough projectives, we must modify this approach.

For an $\fS$-module $V$, let $\uHom(V, -)$ be the right adjoint of the endofunctor $- \otimes V$ of $\Rep(\fS)$. This adjoint exists since $\Rep(\fS)$ is a Grothendieck abelian category and $- \otimes V$ commutes with colimits. Thus, if $U$ and $W$ are two other $\fS$-modules, then we have a natural isomorphism
\begin{displaymath}
\Hom_{\fS}(U, \uHom(V, W)) = \Hom_{\fS}(U \otimes V, W).
\end{displaymath}
Taking $U=\VV_n=\Ind_{\fS(n)}^{\fS}(k)$, one sees that
\begin{displaymath}
\uHom(V, W)^{\fS(n)} = \Hom_{\fS(n)}(V, W).
\end{displaymath}
Since $\uHom(V, W)$ is a smooth representation, it is the union of its $\fS(n)$-invariants. Thus we see that $\uHom(V, W)$ consists of all linear maps $V \to W$ that commute with some $\fS(n)$. Taking $n=0$ in the above equation, we have
\begin{displaymath}
\Hom_{\fS}(V, W) = \Gamma(\uHom(V, W)),
\end{displaymath}
where $\Gamma$ denotes $\fS$-invariants.

A key point is that the functor $\uHom(\VV_n, -)$ is exact, or, more generally, $\uHom(V, -)$ is exact if $V$ is a finite direct sum of $\VV_n$'s. Indeed, suppose that
\begin{displaymath}
0 \to W_1 \to W_2 \to W_3 \to 0
\end{displaymath}
is an exact sequence $\fS$-modules. We must show that the sequence remains exact after applying $\uHom(V, -)$. Since $\uHom(V, -)$ is a right adjoint, it is left exact, so it suffices to verify we have a surjection on the right. The previous paragraph shows that $\uHom(V, -)$ commutes with direct limits, so we can assume that $W_i$'s are finite length. Suppose we have $f \in \uHom(V, W_3)$, i.e., a map $f \colon V \to W_3$ commuting with some $\fS(n)$. Replacing $\fS$ with $\fS(m)$ for $m$ sufficiently large, we can assume the $W_i$'s are semi-induced (Proposition~\ref{prop:shift}) and $f$ is $\fS$-equivariant. Now pull-back the above short exact sequence along $f$ to obtain an extension of $V$ by $W_1$. Since $\Ext^1_{\fS}(V, W_1)=0$ (Proposition~\ref{prop:Ext-semi-induced-sinf}), the sequence splits, which allows us to lift $f$ to $W_2$. We thus see that
\begin{displaymath}
\uHom(V, W_2) \to \uHom(V, W_3)
\end{displaymath}
is surjective, as required.

For $A$-modules $M$ and $M'$, let $\uHom_A(M, M')$ be the space of all $\vert A \vert$-linear maps $M \to M'$ that commute with some $\fS(n)$. We have
\begin{displaymath}
\uHom_A(A \otimes V, M) = \uHom(V, M).
\end{displaymath}
In particular, the functor $\uHom_A(A \otimes \VV_n, -)$ is exact. Call an $A$-module $P$ \defn{quasi-projective} if it is a finite direct sum of modules of the form $A \otimes \VV_n$ $\uHom_A(P, -)$. Since $M$ is finitely generated and $A$ is noetherian, we can resolve $M$ by quasi-projective modules. We can use this resolution to compute both $\rR \uHom_A(M, -)$ and $- \otimes_A^{\rL} M$. This yields a natural isomorphism
\begin{displaymath}
\rR \uHom_A(M, N) = \rR \uHom_B(B \otimes_A^{\rL} M, N),
\end{displaymath}
Applying $\rR \Gamma$ gives the stated result.
\end{proof}

\begin{proposition} \label{prop:vanish3}
Let $M$ be a finitely generated $\cQ_s$-filtered $R/\fh_s$-module, let $T$ be a torsion $R/\fh_s$-module, and let $\fa$ be an $\fS$-ideal of $R$ contained in $\fh_s$. Then $\Ext^i_{R/\fa}(T, M)=0$ for all $i \ge 0$.
\end{proposition}

\begin{proof}
As in the proof of Theorem~\ref{thm:vanish2}, we reduce to the case where $T$ is finitely generated, and then further to the case where it is killed by $\fh_{s-1}$. We now apply derived tensor-Hom adjunction to the ring map $R/\fa \to R/\fh_s$ to obtain
\begin{displaymath}
\rR \Hom_{R/\fa}(T, M) = \rR \Hom_{R/\fh_s}(R/\fh_s \otimes_{R/\fa}^{\rL} T, M).
\end{displaymath}
Since $T$ is killed by $\fh_{s-1}$, so are all the Tor groups in the derived tensor product, and so the result follows from Theorem~\ref{thm:vanish2}.
\end{proof}

\section{From \texorpdfstring{$\FI$}{FI}-modules to \texorpdfstring{$R$}{R}-modules generically} \label{s:Psi}

The functor $\Phi_s$ defined in \S \ref{s:Phi} induces a functor on generic categories, as we explain below. The purpose of this section is to study this induced functor. The $\Ext$ vanishing results from \S \ref{s:ext} play an important role in our analysis.

\subsection{The generic cateory}

Since $\fh_s$ is an $\fS$-prime of $R$ (Proposition~\ref{prop:h-prime}), the quotient $R/\fh_s$ is an $\fS$-domain, and so we can consider its generic category $\Mod_{R/\fh_s}^{\gen}$ as in \S \ref{ss:commalg}. We refer to \cite[\S III]{Gabriel} as a general reference for Serre quotient categories. The category $\Mod_{R/\fh_s}^{\gen}$ is a locally noetherian Grothendieck abelian category. We let
\begin{displaymath}
T \colon \Mod_{R/\fh_s} \to \Mod_{R/\fh_s}^{\gen}
\end{displaymath}
be the quotient functor, which is exact and commutes with colimits. We let $S$ be its right adjoint. This is the same notation that we use for $\FI$-modules, but it should not cause confusion. As with $\FI$-modules, we say that an $R/\fh_s$-module $M$ is \defn{saturated} (resp.\ \defn{derived saturated}) if the natural map $M \to ST(M)$ (resp.\ $\rR ST(M)$) is an isomorphism. A module is saturated if and only if $\Ext^i_{R/\fh_s}(T, M)=0$ for $i=0,1$ and $T$ torsion \cite[\S III.2, Corollaire]{Gabriel}; we note that \cite{Gabriel} uses the term ``$\bC$-ferm\'e'' instead of saturated, where $\bC$ is the torsion category. We have a similar criterion for derived saturation:

\begin{proposition} \label{prop:dersat}
An $R/\fh_s$-module $M$ is derived saturated if and only if $\Ext^i_{R/\fh_s}(T, M)=0$ for all $i \ge 0$ and torsion modules $T$.
\end{proposition}

\begin{proof}
First suppose that $M$ is derived saturated. Let $T(M) \to I^{\bullet}$ be an injective resolution, so that the complex $S(I^{\bullet})$ computes $\rR ST(M)$. By assumption, $M \to S(I^{\bullet})$ is a quasi-isomorphism, and so $S(I^{\bullet})$ is an injective resolution of $M$ (note that $S$ takes injectives to injectives since it is right adjoint to an exact functor). We thus see that $\Ext^i_{R/\fh_s}(T, M)$ is computed by $\Hom_{R/\fh_s}(T, S(I^{\bullet}))$. Each $\Hom$ space here vanishes when $T$ is torsion, since $S(I^i)$ is saturated \cite[\S III.2 Lemme~2]{Gabriel}. Thus we obtain the stated $\Ext$ vanishing.

Next suppose that $\Ext^i_{R/\fh_s}(T, M)=0$ for all $i \ge 0$ and torsion modules $T$. Again, choose an injective resolution $T(M) \to I^{\bullet}$, so that $S(I^{\bullet})$ computes $\rR ST(M)$. Consider a distinguished triangle
\begin{displaymath}
M \to S(I^{\bullet}) \to X \to.
\end{displaymath}
Since the first map is a quasi-isomorphism after applying $T$, we see that $T(X)=0$ in the derived category; thus the cohomology groups of $X$ are torsion modules. Now, suppose $T$ is a torsion $R/\fh_s$-module. By assumption, we have $\rR \Hom_{R/\fh_s}(T, M)=0$. As in the previous paragraph, we have $\rR \Hom_{R/\fh_s}(T, S(I^{\bullet}))=0$. Thus $\rR \Hom_{R/\fh_s}(T, X)=0$ as well. Since $X$ is bounded below and has torsion cohomology, it follows that $X=0$ in the derived category, as required.
\end{proof}

We use an overline to indicate the images of our $\cP$ and $\cQ$ modules in the generic category, e.g., $\ol{\cP}_{s,n}=T(\cP_{s,n})$. We let $\Mod_{R/\fh_s}^{\gen,\cQ}$ be the full subcategory of the generic category generated by the $\cQ$-modules; that is, an object $M$ belongs to $\Mod_{R/\fh_s}^{\gen,\cQ}$ if $M$ is isomorphic to a subquotient of a module of the form $\bigoplus_{i \in I} \ol{\cQ}_{s,n_i}$.

\subsection{The main theorem} \label{ss:genthm}

In \S \ref{s:Phi}, we constructed a functor
\begin{displaymath}
\Phi_s \colon \Mod_{\FI} \to \Mod_{R/\fh_s}.
\end{displaymath}
Since $\Phi_s(\bT_n) = \cQ_{s-1,n}$, it follows that $\Phi_s$ takes torsion $\FI$-modules to torsion $R/\fh_s$-modules; indeed, a (finitely generated) torsion $\FI$-module is an iterated extension of subquotients $\bT_n$'s. We thus see that $\Phi_s$ induces a functor on generic categories. Identifying $\Mod_{\FI}^{\gen}$ with $\Rep(\fS)$ (\S \ref{ss:sym-FI}), we thus have a functor
\begin{displaymath}
\Psi_s \colon \Rep(\fS) \to \Mod_{R/\fh_s}^{\gen}, \qquad \VV_n \mapsto \ol{\cQ}_{s,n}.
\end{displaymath}
By construction, this functor is exact and commutes with arbitrary colimits. The following is the main result of \S \ref{s:Psi}.

\begin{theorem} \label{thm:Psi}
The functor $\Psi_s$ induces an equivalence of $\Rep(\fS)$ with $\Mod_{R/\fh_s}^{\gen,\cQ}$.
\end{theorem}

The proof is contained in the following series of lemmas.

\begin{lemma} \label{lem:genthm-1}
If $M$ is a saturated $\FI$-module then $\Phi_s(M)$ is a saturated $R/\fh_s$-module.
\end{lemma}

\begin{proof}
First suppose that $M$ is finitely generated. Choose an exact sequence $0 \to M \to I_0 \to I_1$ with $I_0$ and $I_1$ semi-induced; this is possible because $M$ is saturated (Proposition~\ref{prop:FI-sec}). We then have an exact sequence
\begin{displaymath}
0 \to \Phi_s(M) \to \Phi_s(I_0) \to \Phi_s(I_1).
\end{displaymath}
Since $\Phi_s$ is exact and carries induced modules to $\cQ_s$-modules, it carries semi-induced modules to $\cQ_s$-filtered modules. Thus $\Phi_s(I_0)$ and $\Phi_s(I_1)$ are $\cQ_s$-filtered. By Theorem~\ref{thm:vanish2}, it follows that these modules are saturated. The kernel of a map of saturated modules is saturated, and so the result follows.

We now treat the general case. If $N$ is a finitely generated subobject of $M$ then $ST(N)$ is a finitely generated saturated subobject of $ST(M)=M$ (finite generation comes from Proposition~\ref{prop:FI-sec}). It follows that $M=\varinjlim M_i$, where the $M_i$ are finitely generated saturated $\FI$-modules. Since $\Psi_s$ commutes with colimits, we have $\Psi_s(M) = \varinjlim \Psi_s(M_i)$. Each $\Psi_s(M_i)$ is saturated by the previous paragraph. Since $\Mod_{R/\fh_s}$ is locally noetherian, a direct limit of saturated objects is saturated (see \cite[\S III.4 Cor.~1]{Gabriel}, and its proof). The result thus follows.
\end{proof}

\begin{lemma} \label{lem:Psi-2}
The functor $\Psi_s$ is fully faithful.
\end{lemma}

\begin{proof}
The functor $\Psi_s$ is identified with the composition
\begin{displaymath}
\xymatrix{
\Mod_{\FI}^{\gen} \ar[r]^S &
\Mod_{\FI} \ar[r]^-{\Phi_s} &
\Mod_{R/\fh_s} \ar[r]^T &
\Mod_{R/\fh_s}^{\gen} }
\end{displaymath}
The functor $S$ is fully faithful, and its image consists of saturated objects \cite[\S III.2]{Gabriel}. The functor $\Phi_s$ is fully faithful on saturated objects (Theorem~\ref{thm:Phi}), and maps saturated objects to saturated objects (Lemma~\ref{lem:genthm-1}). Finally, $T$ is fully faithful on saturated objects, by general properties of Serre quotients \cite[\S III.2]{Gabriel}.
\end{proof}

\begin{lemma} \label{lem:genthm-3}
Let $M$ be a non-zero object of $\Mod_{R/\fh_s}^{\gen}$. Then there exists a non-zero map $\ol{\cQ}_{s,n} \to M$ for some $n$.
\end{lemma}

\begin{proof}
Every $R/\fh_s$-module is a quotient of a direct sum of $\cP_{s,n}$'s. We can thus find a non-zero map $f \colon \ol{\cP}_{s,n} \to M$. Now, we have a filtration $0=F_0 \subset \cdots \subset F_r=\cP_{s,n}$ where each $F_i/F_{i-1}$ is isomorphic to $\cQ_{s,n}$ (Proposition~\ref{prop:PQfilt}), and so we obtain a similar filtration $\ol{F}_{\bullet}$ on $\ol{\cP}_{s,n}$. Let $1 \le i \le r$ be maximal such that $\ol{F}_{i-1}$ is contained in the kernel of $f$. Then $f$ induces a non-zero map $\ol{F}_i/\ol{F}_{i-1} \to M$, as required.
\end{proof}

\begin{lemma} \label{lem:genthm-4}
The functor $\Psi_s$ maps simples to simples.
\end{lemma}

\begin{proof}
Let $L'$ be a simple object of $\Rep(\fS)$. We claim $L=\Psi_s(L')$ is simple. Note that $L$ is non-zero since $\Psi_s$ is faithful. Suppose $M$ is a non-zero subobject of $L$. There is then a non-zero map $f \colon \ol{\cQ}_{s,n} \to L$, for some $n$, with image contained in $M$ (Lemma~\ref{lem:genthm-3}). Since $\Psi_s$ is full and $\ol{\cQ}_{s,n}=\Psi_s(\VV_n)$, we have $f=\Psi_s(f')$ for some $f' \colon \VV_n \to L'$. Since $f$ is non-zero, so is $f'$; since $L'$ is simple it follows that $f'$ is surjective; and since $\Psi_s$ is exact, it follows that $f$ is surjective. Thus $M=L$, which shows that $L$ is simple.
\end{proof}

The lemma implies that $\Psi_s$ maps finite length objects to finite length objects. In particular, $\ol{\cQ}_{s,n}$ has finite length.

\begin{lemma} \label{lem:genthm-5}
Let $\cC$ and $\cD$ be finite lenth abelian categories and let $F \colon \cC \to \cD$ be a fully faithful exact functor that induces an isomorphism on (isomorphism classes of) simple objects. Let $M$ be a finite length object of $\cC$. Then the function
\begin{displaymath}
F \colon \{ \text{subobjects of $M$} \} \to \{ \text{subobjects of $F(M)$} \}
\end{displaymath}
is bijective.
\end{lemma}

\begin{proof}
This is clearly true for simple subobjects of $F(M)$. Now suppose $N'$ is an arbitrary non-zero subobject of $F(M)$. Let $L' \subset N'$ be simple. Then $L'=F(L)$ for some simple $L \subset M$. We thus see that $N'/L'$ is a subobject of $F(M/L)$. By induction on length, there is a subobject $\ol{N}$ of $M/L$ such that $N'/L'=F(\ol{N})$. Letting $N$ be the inverse image of $\ol{N}$ in $M$, we have $N'=F(N)$, as required.
\end{proof}

\begin{lemma} \label{lem:genthm-6}
The essential image of $\Psi_s$ is $\Mod_{R/\fh_s}^{\gen,\cQ}$.
\end{lemma}

\begin{proof}
If $V$ is an object of $\Rep(\fS)$ then $V$ is a quotient of a (possibly infinite) direct sum of $\VV_n$'s. Since $\Psi_s$ is exact and commutes with colimits, it follows that $\Psi_s(V)$ is a quotient of a direct sum of $\ol{\cQ}_{s,n}$'s. Thus the essential image of $\Psi_s$ is contained in $\Mod_{R/\fh_s}^{\gen,\cQ}$. We must now establish the reverse inclusion.

Let $M$ be a finitely length object of $\Mod_{R/\fh_s}^{\gen,\cQ}$. Write $M=Y/X$, where $X \subset Y \subset Z$, and $Z$ is a finite direct sum of $\ol{\cQ}_{s,n}$'s. We have $Z=\Psi_s(Z')$, where $Z'$ is a finite direct sum of $\VV_n$'s. We apply Lemma~\ref{lem:genthm-5} with $\cC$ the category of finite length objects in $\Rep(\fS)$ and $\cD$ the category of finite length objects in $\Mod_{R/\fh_s}^{\gen,\cQ}$ and $F=\Psi_s$. We thus have subrepresentations $X' \subset Y' \subset Z'$ such that $Y=\Psi_s(Y')$ and $X=\Psi_s(X')$. Therefore $M=\Psi_s(Y'/X')$, as required.

Now suppose $M$ is a general object. Write $M=\varinjlim M_i$ where each $M_i$ is finitely length. Then $M_i=\Psi_s(V_i)$ for some $V_i$ by the previous paragraph. Since $\Psi_s$ is fully faithful, the $V_i$'s form an inductive system in $\Rep(\fS)$; let $V$ be the direct limit. Since $\Psi_s$ commutes with colimits, we have $M=\Psi_s(V)$, as required.
\end{proof}

\begin{remark}
The proof shows that every object of $\Mod_{R/\fh_s}^{\gen,\cQ}$ is a quotient of a sum of $\ol{\cQ}_{s,n}$'s.
\end{remark}

\subsection{The adjoint}

Since $\Psi_s$ is a cocontinuous functor of Grothendieck abelian categories, it admits a right adjoint
\begin{displaymath}
\Psi_s^* \colon \Mod_{R/\fh_s}^{\gen} \to \Rep(\fS).
\end{displaymath}
Since $\Psi_s$ is fully faithful, the unit $V \to \Psi_s^*(\Psi_s(V))$ is an isomorphism for any $\fS$-module $V$.

\begin{proposition} \label{prop:Psi-counit}
Let $M$ be a finite length object of $\Mod_{R/\fh_s}^{\gen}$. Then the co-unit
\begin{displaymath}
\Psi_s(\Psi_s^*(M)) \to M
\end{displaymath}
is injective, and its image is the maximal subobject of $M$ that belongs to $\Mod_{R/\fh_s}^{\gen,\cQ}$.
\end{proposition}

\begin{proof}
Let $M'$ be the sum of all subobjects of $M$ belonging to $\Mod_{R/\fh_s}^{\gen,\cQ}$. Since $\Mod_{R/\fh_s}^{\gen,\cQ}$ is closed under quotients, it follows that $M'$ belongs to $\Mod_{R/\fh_s}^{\gen,\cQ}$, and it is clearly the maximal subobject of $M$ in this category. Write $M'=\Psi_s(V)$ for some $V$ in $\Rep(\fS)$, and let $i \colon \Psi_s(V) \to M$ be the inclusion. Now, suppose $f \colon \Psi_s(W) \to M$ is an arbitrary map, with $W$ in $\Rep(\fS)$. Since the image of $f$ belongs to $\Mod_{R/\fh_s}^{\gen,\cQ}$, it follows that $f$ factors uniquely as $i \circ g$, where $g \colon \Psi_s(W) \to \Psi_s(V)$. Since $\Psi_s$ is fully faithful, $g=\Psi_s(g')$ for a unique $g'$. It thus follows that $\Psi_s^*(M)=V$, and the co-unit $\Psi_s(\Psi_s^*(M)) \to M$ is identified with $i$. This completes the proof.
\end{proof}

\begin{proposition} \label{prop:Psi-star}
We have $\Psi_s^*(\ol{\cP}_{s,n}) = \VV_n$, and, more generally, $\Psi_s^*(\ol{\cP}_s(V))=\PP(V)$.
\end{proposition}

\begin{proof}
Let $V$ be an object of $\Rep(\fS)$, and write $V=\Theta(N)$ for an $\FI$-module $N$ (see \S \ref{ss:sym-FI}). Note that $\Psi_s(V)=T(\Phi_s(N))$, by definition. We have isomorphisms
\begin{align*}
\Hom_{\fS}(V, \Psi_s^*(\ol{\cP}_{s,n}))
&=\Hom(\Psi_s(V), \ol{\cP}_{s,n})
=\Hom_{R/\fh_s}(\Phi_s(N), \cP_{s,n}) \\
&=\Hom_{\FI}(N, \PP_n)
=\Hom_{\fS}(V, \VV_n)
\end{align*}
In the first step, we used the adjunction $(\Psi_s, \Psi_s^*)$; in the second step, we used the adjunction $(T, S)$ and the fact that $\cP_{s,n}$ is saturated (Theorem~\ref{thm:vanish2}); in the third step, we used the adjunction $(\Phi_s, \Phi_s^*)$ and the computation of $\Phi_s^*(\cP_{s,n})$ (Proposition~\ref{prop:Phi-star}); and in the fourth step, we used in the adjunction $(T, S)$ for $\FI$, the fact that $\PP_n$ is saturated (\S \ref{ss:fi-sat}), and the identification of $\Mod_{\FI}^{\gen}$ with $\Rep(\fS)$ (\S \ref{ss:sym-FI}). The first statement thus follows from Yoneda's lemma. The second statement follows from an argument as in the second paragraph of the proof of Lemma~\ref{lem:Phi-1}, though we must use a co-presentation of $V$ instead of a presentation since $\Psi_s^*$ is only left exact.
\end{proof}

\begin{proposition}  \label{prop:vanish4}
We have $\Ext^i(\ol{\cQ}_{s,m}, \ol{\cP}_{s,n})=0$ for all $n$ and $m$ and $i>0$.
\end{proposition}

\begin{proof}
By derived adjunction (\S \ref{ss:deradj}), we have
\begin{displaymath}
\rR \Hom_{R/\fh_s}(M, \rR S(N)) = \rR \Hom(T(M), N)
\end{displaymath}
for $M$ in $\Mod_{R/\fh_s}$ and $N$ in $\Mod_{R/\fh_s}^{\gen}$. We apply this with $M=\cQ_{s,m}$ and $N=\ol{\cP}_{s,n}$. We have $\rR S(\ol{\cP}_{s,n})=\cP_{s,n}$ by Theorem~\ref{thm:vanish2} and Proposition~\ref{prop:dersat}, and so the result follows from Theorem~\ref{thm:vanish1}.
\end{proof}

\begin{proposition}\label{prop:RPsi-star}
We have $\rR^i \Psi_s^*(\ol{\cP}_{s,n})=0$ for $i>0$.
\end{proposition}

\begin{proof}
By derived adjunction (\S \ref{ss:deradj}), for an $\fS$-module $V$ and a generic $R/\fh_s$-module $M$, we have a natural isomorphism
\begin{displaymath}
\rR \Hom_{\fS}(V, \rR \Psi_s^*(M)) = \rR \Hom(\Psi_s(V), M).
\end{displaymath}
We apply this with $V=\VV_m$ and $M=\ol{\cP}_{s,n}$. The right side only has cohomology in degree~0 by the Proposition~\ref{prop:vanish4}. We thus have $\rR^i \Psi_s^*(\ol{\cP}_{s,n})=0$ for $i>0$ by the following lemma.
\end{proof}

\begin{lemma}
Let $X$ be a complex in $\Rep(\fS)$ supported in non-negative degrees. Suppose that $\rR^i \Hom_{\fS}(\VV_m, X)=0$ for all $i>0$ and all $m \ge 0$. Then $\rH^i(X)=0$ for all $i>0$.
\end{lemma}

\begin{proof}
Replace $X$ with a quasi-isomorphic complex of injectives. Suppose by way of contradiction that $\rH^i(X) \ne 0$ for some $i>0$. Choose a non-zero element of $\rH^i(X)$, and lift it to an element of $x$ of $\ker(d^i)$, where $d^i \colon X^i \to X^{i+1}$ is the differential. Let $m$ be such that $x$ is fixed by $\fS(m)$. We then have a map $f \colon \VV_m \to \ker(d^i)$ in $\Rep(\fS)$ that takes $e_{1,\ldots,m}$ to $x$ (Proposition~\ref{prop:Vn}). This defines a map of complexes $f \colon \VV_m[-i] \to X$, where $\VV_m[-i]$ is the complex that is $\VV_m$ in degree~$i$, and~0 elsewhere. Since $f$ induces a non-zero map on $\rH^i$, it defines a non-zero element of $\rR^i \Hom_{\fS}(\VV_m, X)$, a contradiction. The result thus follows.
\end{proof}

\section{The generic category}

In this section we prove our main results on the generic category of $R/\fh_s$-modules.

\subsection{Finite length}

The following is the most fundamental result about the generic category.

\begin{theorem} \label{thm:gen-fin-len}
The category $\Mod_{R/\fh_s}^{\gen}$ is locally of finite length: if $M$ is a finitely generated $R/\fh_s$-module then $T(M)$ has finite length.
\end{theorem}

\begin{proof}
Theorem~\ref{thm:Psi} implies that $\ol{\cQ}_{s,n}=\Psi_s(\VV_n)$ has finite length. Since $\cP_{s,n}$ has a finite length filtration with graded pieces $\cQ_{s,n}$ (Proposition~\ref{prop:PQfilt}), it follows that $\ol{\cP}_{s,n}$ has finite length. If $M$ is a finitely generated $R/\fh_s$-module then $M$ is a quotient of a finite sum of $\cP_{s,n}$'s, and so $T(M)$ has finite length.
\end{proof}

\subsection{Simple objects}

Recall that, for an abelian category $\cA$, let $\Irr(\cA)$ denotes the set of isomorphism classes of simple objects. We now aim to describe the simple objects in $\Mod_{R/\fh_s}^{\gen}$. We begin with an easy result.

\begin{theorem}
We have a bijection
\begin{displaymath}
\Irr(\Rep(\fS)) \to \Irr(\Mod_{R/\fh_s}^{\gen}), \qquad
V \mapsto \Psi_s(V).
\end{displaymath}
\end{theorem}

\begin{proof}
Since every object of $\Mod_{R/\fh_s}^{\gen}$ is a quotient of a sum of $\ol{\cP}_{s,n}$'s, it follows that any simple object is a constituent of some $\ol{\cP}_{s,n}$. Since $\cP_{s,n}$ has a filtration with graded pieces $\cQ_{s,n}$ (Proposition~\ref{prop:PQfilt}), it follows than any simple object is a constituent of some $\ol{\cQ}_{s,n}$, and therefore belongs to the category $\Mod_{R/\fh_s}^{\gen,\cQ}$. The result thus follows from Theorem~\ref{thm:Psi}.
\end{proof}

The above theorem gives a description of the simples of the generic category using the functor $\Psi_s$ and simples of $\Rep(\fS)$. We can derive from this a more intrinsic characterization of the simples.

\begin{theorem} \label{thm:simples}
If $V$ is an irreducible representation of $\fS_n$ then $\soc(\ol{\cQ}_s(V)) = \soc(\ol{\cP}_s(V))$ is simple. Moreover, we have a bijection
\begin{displaymath}
\coprod_{n \ge 0} \Irr(\Rep(\fS_n)) \to \Irr(\Mod_{R/\fh_s}^{\gen}), \qquad
V \mapsto \soc(\ol{\cQ}_s(V)).
\end{displaymath}
In particular, if $\chr(k)=0$ then the simples of $\Mod_{R/\fh_s}^{\gen}$ have the form $\soc(\ol{\cQ}_{s,\lambda})$.
\end{theorem}

\begin{proof}
Since the simples of $\Rep(\fS)$ are exactly of the form $\soc(\PP(V))$ with $V$ a simple $k[\fS_n]$-module (Proposition~\ref{prop:FI-gen-simple} and \S \ref{ss:sym-FI}), the statement with $\cQ_s$ follows immediately from the previous corollary. It remains to show that $\soc(\ol{\cP}_s(V)) = \soc(\ol{\cQ}_s(V))$ when $V$ is a simple $k[\fS_n]$-module. Since $\soc(\ol{\cP}_s(V))$ is a sum of simple objects of $\Mod_{R/\fh_s}^{\gen}$ and all simple objects belong to $\Mod_{R/\fh_s}^{\gen,\cQ}$, it follows that $\soc(\ol{\cP}_s(V))$ belongs to $\Mod_{R/\fh_s}^{\gen,\cQ}$. Therefore the inclusion $\soc(\ol{\cP}_s(V)) \to \ol{\cP}_s(V)$ factors through $\Psi_s(\Psi_s^*(\ol{\cP}_s(V)))$ (Proposition~\ref{prop:Psi-counit}), which is identified with $\ol{\cQ}_s(V)$ (Proposition~\ref{prop:Psi-star}), from which the result follows.
\end{proof}

\subsection{Derived generators} \label{ss:dergen}

Let $\cA$ be a locally noetherian Grothendieck abelian category. We let $\rD^b_{\rm fg}(\cA)$ be the full subcategory of the derived category $\rD(\cA)$ spanned by objects $X$ such that $\rH^i(X)$ is finitely generated (i.e., noetherian) for all $i$, and zero for all but finitely many $i$. Given a collection $S$ of objects of $\rD^b_{\rm fg}(\cA)$, the triangulated subcategory \defn{generated} by $S$ is the intersection of all replete full triangulated subcategories of $\rD^b_{\rm fg}(\cA)$ containing $S$. (Recall that a full subcategory is \defn{replete} if it contains all objects in the ambient category isomorphic to an object in the subcategory.) Having a set $S$ of generators for $\rD^b_{\rm fg}(\cA)$ can be very useful; for instance, $S$ will generate the Grothendieck group, and many kinds of homological questions about objects in $\cA$ can be reduced to $S$. We now determine generators for our generic category.

\begin{theorem} \label{thm:Q-gen}
The objects $\ol{\cQ}_s(V)$, with $V$ an irreducible $\fS_n$-representation (with $n$ variable), generate $\rD^b_{\rm fg}(\Mod_{R/\fh_s}^{\gen})$.
\end{theorem}

\begin{proof}
Let $\cK$ be the subcategory generated by the $\ol{\cQ}_s(V)$. Clearly, $\cK$ contains any finite length $\ol{\cQ}_s$-filtered module. If $M$ is a finite length object of $\Mod^{\gen,\cQ}_{R/\fh_s}$ then, by Proposition~\ref{prop:sym-semi-ind-res} and Theorem~\ref{thm:Psi},  there is a finite resolution
\begin{displaymath}
0 \to M \to I^0 \to \cdots \to I^n \to 0
\end{displaymath}
where each $I^i$ is $\ol{\cQ}_s$-filtered; thus $M$ belongs to $\cK$. Since any finite length object $M$ of $\Mod_{R/\fh_s}^{\gen}$ has a finite length filtration where the graded pieces belong to $\Mod_{R/\fh_s}^{\gen,\cQ}$, it follows that $M$ belongs to $\cK$. Since the finite length modules in $\Mod_{R/\fh_s}^{\gen}$ clearly generate $\rD^b_{\rm fg}(\Mod_{R/\fh_s}^{\gen})$, the result thus follows.
\end{proof}

\begin{question}
In the above proof, we saw that any finite length object of $\Mod_{R/\fh_s}^{\gen,\cQ}$ admits a finite resolution by $\ol{\cQ}_s$-filtered objects. Does this hold more generally for any finite length object of $\Mod_{R/\fh_s}^{\gen}$?
\end{question}

\begin{remark} \label{rmk:P-dont-gen}
Suppose $s \ge 1$. We have
\begin{displaymath}
\Ext^i(\ol{\cQ}_{s,1}, \ol{\cQ}_{s,1}) = \Ext^i_{R/\fh_s}(\cQ_{s,1}, \cQ_{s,1}) = k
\end{displaymath}
for all $i \ge 0$, where the first step uses that $\cQ_{s,1}$ is derived saturated, and the second uses Remark~\ref{rmk:Q11-ext}. We thus see that $\ol{\cQ}_{s,1}$ does not have finite injective dimension in $\Mod_{R/\fh_s}^{\gen}$. We will see in Theorem~\ref{thm:P-inj} below that, in characteristic~0, the $\cP_{s,\lambda}$'s are injective in $\Mod_{R/\fh_s}^{\gen}$. It thus follows that the $\cP_s(V)$'s do not generated $\rD^b_{\rm fg}(\Mod_{R/\fh_s}^{\gen})$, at least in characteristic~0.
\end{remark}

\subsection{The Grothendieck group} \label{ss:gen-K}

For a locally noetherian abelian category $\cA$, let $\rK(\cA)$ denote the Grothendieck group of the category of noetherian objects in $\cA$. Put $K=\rK(\Mod_{R/\fh_s}^{\gen})$. We now investigate the structure of this group.

\begin{proposition} \label{prop:Groth-Q}
We have an isomorphism of abelian groups
\begin{displaymath}
\bigoplus_{n \ge 0} \rK(\Rep(\fS_n)) \to K, \qquad [V] \mapsto [\ol{\cQ}_s(V)].
\end{displaymath}
In particular, in if $\chr(k)=0$ then the classes $[\ol{\cQ}_{s,\lambda}]$ form a $\bZ$-basis of $K$
\end{proposition}

\begin{proof}
Since the simple objects of $\Mod_{R/\fh_s}^{\gen}$ all live in the subcategory $\Mod_{R/\fh_s}^{\gen,\cQ}$, it follows that their Grothendieck groups agree. The result thus follows from Theorem~\ref{thm:Psi}, Proposition~\ref{prop:gen-FI-K}, and \S \ref{ss:sym-FI}.
\end{proof}

We now aim to establish a similar result, but with $P$'s instead of $Q$'s. We require the following two lemmas.

\begin{lemma} \label{lem:Groth-P-1}
Let $V$ be a $k[\fS_n]$-module and let $S=k[x_1, \ldots, x_n]/(x^{s+1})$. Then we have
\begin{displaymath}
[\ol{\cP}_s(V)] = [\ol{\cQ}_s(S \otimes V)]
\end{displaymath}
in $K$, where here we are simply treating $S$ as a $k[\fS_n]$-module.
\end{lemma}

\begin{proof}
In Proposition~\ref{prop:PQfilt2}, we showed that $\cP_s(V)$ is a $\cQ_s$-filtered module. Examining the proof, we see that $\cP_s(V)$ has a filtration where the graded pieces are $\cQ_s(\fm^i/\fm^{i+1} \otimes V)$, where $\fm$ is the maximal ideal of $S$. Of course, $\cQ_s(S \otimes V)$ also has a filtration with the same graded pieces, and so the result follows.
\end{proof}

\begin{lemma} \label{lem:Groth-P-2}
Letting $S$ be as above, $[S]$ is a unit of the ring $\bQ \otimes \rK(\Rep(\fS_n))$.
\end{lemma}

\begin{proof}
First suppose $\chr(k)=0$; without loss of generality, we take $k=\bC$. If $V$ is a $k[\fS_n]$-module then $[V]$ is a unit of $\bQ \otimes \rK(\Rep(\fS_n))$ if and only if it is a unit of $\bC \otimes \rK(\Rep(\fS_n))$. The latter is identified, as ring, with the space of class functions on $\fS_n$, and so we see that $[V]$ is a unit if and only if the character $\chi_V$ is nowhere vanishing. Now, $S$ is a permutation representation of $\fS_n$, and so $\chi_S(g)$ is the number of fixed points of $g$ on the standard basis. This is always positive since $1 \in S$ is fixed. The proof in positive characteristic is similar, but using the Brauer character.
\end{proof}

We can now obtain a description of $K$ in terms of the $P$'s.

\begin{proposition} \label{prop:Groth-P}
We have an isomorphism of $\bQ$-vector spaces
\begin{displaymath}
\bigoplus_{n \ge 0} \bQ \otimes \rK(\Rep(\fS_n)) \to \bQ \otimes K, \qquad [V] \mapsto [\ol{\cP}_s(V)].
\end{displaymath}
In particular, if $\chr(k)=0$ then the classes $[\ol{\cP}_{s,\lambda}]$ form a $\bQ$-basis of $\bQ \otimes K$.
\end{proposition}

\begin{proof}
Define
\begin{displaymath}
q_n \colon \bQ \otimes \rK(\Rep(\fS_n)) \to \bQ \otimes K, \qquad [V] \mapsto [\ol{\cQ}_s(V)],
\end{displaymath}
and define $p_n$ similarly but using $\cP_s(V)$. Let $S_n=k[x_1, \ldots, x_n]/(x_i^{s+1}))$, and let
\begin{displaymath}
\mu_n \colon \rK(\Rep(\fS_n)) \to \rK(\Rep(\fS_n))
\end{displaymath}
be the multiplication-by-$[S_n]$ map. Lemma~\ref{lem:Groth-P-1} shows that $p_n=q_n \circ \mu_n$, while Lemma~\ref{lem:Groth-P-2} shows that $\mu_n$ is an isomorphism. Since $\bigoplus_{n \ge 0} q_n$ is an isomorphism by Proposition~\ref{prop:Groth-Q}, the same holds with the $p$'s, which completes the proof.
\end{proof}

Let $\Lambda=\rK(\Rep(\fS))$. In characteristic~0, $\Lambda$ is identified with the ring of symmetric functions. Since $\Mod_{R/\fh_s}^{\gen}$ is a module category for $\Rep(\fS)$, it follows that $K$ is naturally a $\Lambda$-module; explicitly, scalar multiplication is defined by $[V] \cdot [M] = [V \otimes M]$. We now determine the structure of $K$ as a $\Lambda$-module, at least rationally.

\begin{proposition} \label{prop:K-gen}
The $(\bQ \otimes \Lambda)$-module $\bQ \otimes K$ is free of rank one with basis $[\ol{\cP}_{s,0}]$.
\end{proposition}

\begin{proof}
Consider the map
\begin{displaymath}
i \colon \bQ \otimes \Lambda \to \bQ \otimes K, \qquad i([V]) = [V] \cdot [\ol{\cP}_{s,0}].
\end{displaymath}
We must show that $i$ is an isomorphism. We have $i([\PP(V)]) = [\ol{\cP}_s(V)]$. As $V$ varies over simple $k[\fS_n]$-modules, the classes $[\PP(V)]$ form a basis of $\bQ \otimes \Lambda$ (Proposition~\ref{prop:gen-FI-K} and \S \ref{ss:sym-FI}), and the classes $[\ol{\cP}_s(V)]$ form a basis of $\bQ \otimes K$ (Proposition~\ref{prop:Groth-P}). The result follows.
\end{proof}

\subsection{Injective objects}

We now determine the injective objects of the generic category, in characteristic~0.

\begin{theorem} \label{thm:P-inj}
If $\chr(k)=0$ then the $\ol{\cP}_{s,\lambda}$ are the indecomposable injectives of $\Mod_{R/\fh_s}^{\gen}$.
\end{theorem}

\begin{proof}
Let $V$ be an $\fS$-module and let $M$ be a generic $R/\fh_s$-module. By derived adjunction (\S \ref{ss:deradj}), we have a natural isomorphism
\begin{displaymath}
\rR \Hom_{\fS}(V, \rR \Psi_s^*(M)) = \rR \Hom(\Psi_s(V), M).
\end{displaymath}
We apply this with $M=\ol{\cP}_{s,n}$. Since $\rR \Psi_s^*(\ol{\cP}_{s,n})=\VV_n$ (Propositions~\ref{prop:Psi-star} and~\ref{prop:RPsi-star}), we find
\begin{displaymath}
\Ext^i_{\fS}(V, \VV_n) = \Ext^i(\Psi_s(V), \ol{\cP}_{s,n}).
\end{displaymath}
Since $\VV_n$ is injective in characteristic~0 (Remark~\ref{rmk:FI-gen-inj}), the higher $\Ext$ groups vanish here. Since every simple of $\Mod_{R/\fh_s}^{\gen}$ has the form $\Psi_s(V)$ for appropriate $V$, we see that $\ol{\cP}_{s,n}$ is injective. Since $\ol{\cP}_{s,\lambda}$ is a summand of $\ol{\cP}_{s,n}$, it too is injective. Since the socles of the $\ol{\cP}_{s,\lambda}$'s are the simple objects (Theorem~\ref{thm:simples}), it follows that they are the indecomposable injectives.
\end{proof}

\begin{remark}
Since $\cP_{s,\lambda}$ is saturated (Theorem~\ref{thm:vanish2}), we see that it is an injective module in $\Mod_{R/\fh_s}$.
\end{remark}

\subsection{Explicit description} \label{ss:char0}

Assume for the duration of \S \ref{ss:char0} that $\chr(k)=0$, and fix $s \ge 0$. We aim to give an explicit description of the generic category. An \defn{$\FI$-algebra} is a functor from $\FI$ to the cateory of commutative $k$-algebras. Suppose $A$ is an $\FI$-algebra. Then an \defn{$A$-module} is an $\FI$-module $M$ for which each $M(S)$ is equipped with the structure of an $A(S)$-module such that if $i \colon S \to T$ is an injection then $i_* \colon M(S) \to M(T)$ is a map of $A(S)$-modules, where $A(S)$ acts on $M(T)$ through the given map $i_* \colon A(S) \to A(T)$.

Let $\bA_s$ be the $\FI$-algebra define by
\begin{displaymath}
\bA_s(S) = k[x_i]_{i \in S}/(x_i^{s+1}),
\end{displaymath}
with obvious transition maps. We say that an $\bA_s$-module $M$ is \defn{locally finite} if it is the union of its finite length submodules; one easily sees that this is equivalent to $M$ being torsion as an $\FI$-module. We write $\Mod_{\bA_s}^{\rm lf}$ for the category of locally finite $\bA_s$-modules. The following is our main result:

\begin{theorem} \label{thm:char0}
The categories $\Mod_{R/\fh_s}^{\gen}$ and $\Mod_{\bA_s}^{\rm lf}$ are equivalent.
\end{theorem}

Fix an $\FI$-algebra $A$. Define a category $\cC(A)$ as follows. The objects are finite sets. We put
\begin{displaymath}
\Hom_{\cC(A)}(S,T) = A(T) \otimes k[\Hom_{\FI}(S,T)].
\end{displaymath}
That is, this $\Hom$ spaces is a free $A(T)$-module with basis indexed by injections $S \to T$. We denote elements of this $\Hom$ space by $a \otimes f$ where $a \in A(T)$ and $f \colon S \to T$. Suppose $U$ is a third finite set. Then composition is defined by
\begin{displaymath}
(b \otimes g)(a \otimes f) = bg_*(a) \otimes gf.
\end{displaymath}
Here $f \colon S \to T$ and $g \colon T \to U$ are injections, $a \in A(T)$, and $b \in A(U)$. One readily verifies that composition is well-defined and $1 \otimes \id_S$ is the identity for $S$. This category is relevant since it allows us to describe the category of $A$-modules. Define a $\cC(A)$-module to be a $k$-linear functor $\cC(A) \to \Vec$. Then:

\begin{lemma}
The category of $A$-modules is equivalent to the category of $\cC(A)$-modules.
\end{lemma}

\begin{proof}
Suppose $M$ is an $A$-module. Then $M(S)$ is an $A(S)$-module for each finite set $S$, and in particular, a $k$-vector space. Suppose that $a \otimes f \colon S \to T$ is a morphism in $\cC(A)$. Then we have a map $M(S) \to M(T)$ by first applying $f_*$, and then multiplying by $a$, using the $A(T)$-module structure on $M(T)$. One verifies that this construction is compatible with composition in $\cC(A)$, and thus defines a $\cC(A)$-module. One can similarly go back from $\cC(A)$-modules to $A$-modules.
\end{proof}

We now make a few simple comments about $\bA_s$-modules. First, if $V$ is an $\fS_n$-representation then we can regard $V$ as an $\bA_s$-module concentrated in degree $n$, where the variables and $\FI$-transition maps act by~0. In this way, the Specht modules $S^{\lambda}$ are exactly the simple $\bA_s$-modules. Define a $\cC(\bA_s)$-module $\bI_{s,n}$ by
\begin{displaymath}
\bI_{s,n}(S) = \Hom_{\cC(\bA_s)}(S, [n])^*.
\end{displaymath}
This module is injective; see \cite[\S 3.4]{brauercat1}. It also has finite length, as it is supported in degrees $\le n$ and finite dimensional in each degree. We have
\begin{displaymath}
\bI_{s,n}([n]) = k[x_1, \ldots, x_n]/(x_i^{s+1}) \otimes k[\fS_n],
\end{displaymath}
and one easily sees that
\begin{displaymath}
(x_1 \cdots x_n)^s \otimes k[\fS_n],
\end{displaymath}
is the socle of $\bI_{s,n}$. In particular, the simple $\bA_s$-module $S^{\lambda}$ injects into $\bI_{s,n}$ where $n=\vert \lambda \vert$, and so we see that the $\bI_{s,n}$'s provide enough injectives for the category $\Mod_{\bA_s}^{\rm lf}$. Let $\cD$ be the full subcategory of $\Mod_{\bA_s}$ spanned by the $\bI_{s,n}$'s. We have a functor
\begin{displaymath}
\cC(\bA_s)^{\op} \to \cD, \qquad [n] \mapsto \bI_{s,n},
\end{displaymath}
which one easily sees is an equivalence.

Let $\cE$ be the full subcategory of $\Mod_{R/\fh_s}^{\gen}$ spanned by the objects $\ol{\cP}_{s,n}$ with $n \ge 0$. The following lemma is the key point in the proof of the theorem.

\begin{lemma}
We have an equivalence $\cC(\bA_s)^{\op} \to \cE$.
\end{lemma}

\begin{proof}
We define a functor $\cC(\bA_s)^{\op} \to \cE$. On objects, it is defined by $[n] \mapsto \ol{\cP}_{s,n}$. We must now define the functor on morphisms and show that it is fully faithful.

We have
\begin{displaymath}
\Hom_{R/\fh_s}^{\gen}(\ol{\cP}_{s,n}, \ol{\cP}_{s,m})
= \Hom_{R/\fh_s}(\cP_{s,n}, \cP_{s,m})
= \cP_{s,m}^{\fS(n)},
\end{displaymath}
where the first identification follows since $\cP_{s,m}$ is saturated (Theorem~\ref{thm:vanish2}), and the second follows from the mapping property for $\cP_{s,n}$ (see Proposition~\ref{prop:Vn}). Now, we have
\begin{displaymath}
\cP_{s,m}^{\fS(n)} = \sum_{f \colon [m] \to [n]} a(f) e_{f(1), \ldots, f(m)}
\end{displaymath}
where the sum is over injections $f$, and $a(f) \in k[x_1, \ldots, x_n]/(x_i^{s+1})$. We thus have an isomorphism of vector spaces
\begin{displaymath}
\Hom_{\cC(\bA_s)}([m], [n]) \to \cP_{s,m}^{\fS(n)}, \qquad
a \otimes f \mapsto a e_{f(1), \ldots, f(m)}.
\end{displaymath}
We use this isomorphism to define our functor on morphisms. It is straightforward to check compatibility with composition and identity morphisms. The functor is fully faithful by construction. The result thus follows.
\end{proof}

\begin{proof}[Proof of Theorem~\ref{thm:char0}]
$\Mod_{R/\fh_s}^{\gen}$ is a locally finite Grothendieck abelian category and $\cE$ is a class of enough injective objects (by Theorem~\ref{thm:P-inj}). Similarly, $\Mod_{\bA_s}^{\rm lf}$ is a locally finite Grothendieck abelian category and $\cD$ is a class of enough injective objects. It follows that any equivalence $\cD \to \cE$ extends to an equivalence $\Mod_{\bA_s}^{\rm lf} \to \Mod_{R/\fh_s}^{\gen}$. Since we have an equivalence $\cD \cong \cE$, the result follows.
\end{proof}

\section{General modules} \label{s:general}

In this final section, we prove some results on $R/\fh_s$-modules.

\subsection{Saturation and local cohomology} \label{ss:loccoh}

We begin by reviewing a bit more material about Serre quotient categories, following \cite[\S 4]{SSglII}. Let $\cA$ be a Grothendieck abelian category and let $\cB$ be a localizing subcategory. We assume the following condition:
\begin{itemize}
\item[(Inj)] Injective objects of $\cB$ remain injective in $\cA$.
\end{itemize}
As usual, we let $T \colon \cA \to \cA/\cB$ be the quotient functor, and we let $S$ be its right adjoint. We let $\Sigma \colon \cA \to \cA$ be the composition $ST$, which we call the \defn{saturation functor}. Given an object $M$ of $\cA$, there is a maximal subobject $\Gamma(M)$ of $M$ that belongs to $\cB$. This defines a functor $\Gamma \colon \cA \to \cB$ that is right adjoint to the inclusion $\cB \subset \cA$. We call $\rR^i \Gamma(M)$ the $i$th \defn{local cohomology} of $M$. For an object $M$ of $\rD^+(\cA)$, there is a canonical triangle
\begin{displaymath}
\rR \Gamma(M) \to M \to \rR \Sigma(M) \to
\end{displaymath}
by \cite[Proposition~4.6]{SSglII}.

Suppose now that $\cA$ is locally noetherian and that for any object $M$ of $\cA$, each $\rR^i \Sigma(M)$ is finitely generated (i.e., noetherian), and only finitely many are non-zero. We then have a semi-orthogonal decomposition
\begin{displaymath}
\rD^b_{\rm fg}(\cA) = \langle \cD_0, \cD_1 \rangle
\end{displaymath}
where $\cD_0$ is the full subcategory of $\rD^b_{\rm fg}(\cA)$ on objects $M$ with $\rR \Sigma(M)=0$, and $\cD_1$ is the full subcategory on objects $M$ with $\rR \Gamma(M)=0$. Moreover, $\cD_0$ is equivalent to $\rD^b_{\rm fg}(\cB)$ and $\cD_1$ is equivalent to $\rD^b_{\rm fg}(\cA/\cB)$. See \cite[\S 4.3]{SSglII} for these claims, and the definition of semi-orthogonal decomposition.

The above assumptions have one additional notable consequence, namely, the exactness of the sequence
\begin{displaymath}
0 \to \rK(\cB) \to \rK(\cA) \to \rK(\cA/\cB) \to 0.
\end{displaymath}
We always have exactness if we omit the~0 on the left, so the important point is that our hypotheses ensure that $\rK(\cB) \to \rK(\cA)$ is injective. To see this, simply note that $\rR \Gamma$ gives a left inverse. See \cite[\S 4.3]{SSglII} for details.

In the above discussion, we have essentially decomposed $\cA$ into two pieces, namely, $\cB$ and $\cA/\cB$. It is possible to decompose into more pieces if one starts with a chain of localizing subcategories. We will do this below. We refer to \cite[\S 4.2]{SSglII} for a discussion of the general case.

\subsection{Property (Inj)}

We now aim to establish this property in cases of interest. Suppose that $A$ is an $\fS$-algebra and $\fa$ is an $\fS$-ideal of $A$. Then
\begin{displaymath}
\fR_{\fa}(A) = \bigoplus_{n \ge 0} \fa^n t^n
\end{displaymath}
is an $\fS$-algebra, called the \defn{Rees algebra} of $\fa$. Here $t$ is simply a formal symbol. The following is the key result:

\begin{proposition}
The Rees algebra $\fR_{\fh_s}(R)$ is noetherian, for any $s$.
\end{proposition}

\begin{proof}
We have an $\fS$-algebra isomorphism
\begin{displaymath}
k[x_i,y_i]_{i \ge 1} \to \fR_{\fh_s}(R), \qquad
x_i \mapsto x_i, \quad y_i \mapsto x_i^{s+1} t.
\end{displaymath}
The result thus follows from Cohen's theorem that $k[x_i, y_i]$ is $\fS$-noetherian \cite{Cohen2, NagelRomer}.
\end{proof}

\begin{remark}
We do not know if $\fR_\fa(R)$ is noetherian for arbitrary $\fS$-ideals $\fa \subset R$. This is an interesting and important open problem. When $\fa$ is a monomial ideal, an affirmative answer follows from the main result of \cite{DEKL}.
\end{remark}

\begin{corollary}
The Artin--Rees lemma holds for $\fh_s$, in the following sense. Let $M$ be a finitely generated $R$-module and let $N$ be an $R$-submodule of $M$. Then there exists $r$ such that $\fh_s^n M \cap N=\fh_s^{n-r}(\fh_s^r M \cap N)$ for all $n \ge r$.
\end{corollary}

\begin{proof}
The usual proof of the Artin--Rees lemma applies. In essence, $\bigoplus_{n \ge 0} \fh_s^n M$ is a finitely generated $\fR_{\fh_s}(R)$-module and $\bigoplus_{n \ge 0} (\fh_s^n M \cap N)$ is a submodule. By noetherianity, the latter is finitely generated, and is thus generated in degrees $\le r$ for some $r$. This yields the result.
\end{proof}

\begin{corollary} \label{cor:Inj}
Let $\fa$ be an $\fS$-ideal of $R$ contained in $\fh_s$. Let $\cA=\Mod_{R/\fa}$, and let $\cB$ be the subcategory of modules locally annihilated by a power of $\fh_s$. Then $\cB \subset \cA$ satisfies (Inj).
\end{corollary}

\begin{proof}
This follows from \cite[Corollary~4.19]{SSglII}.
\end{proof}

\subsection{Semi-orthogonal decomposition} \label{ss:semi-orth}

Fix $s \ge 1$. We introduce the following notation, where $0 \le r \le s$.
\begin{itemize}
\item $\cC$ is the category $\Mod_{R/\fh_s}$.
\item $\cC_{\le r}$ is the full subcategory of $\cC$ on modules locally annihilated by a power of $\fh_r$.
\item $\cC_{>r}$ is the Serre quotient $\cC/\cC_{\le r}$.
\item $T_{>r} \colon \cC \to \cC_{>r}$ is the quotient functor and $S_{>r} \colon \cC_{>r} \to \cC$ is its right adjoint.
\item $\Sigma_{>r}=S_{>r} \circ T_{>r}$ is the saturation functor with respect to $\cC_{\le r}$.
\item $\Gamma_{\le r} \colon \cC \to \cC_{\le r}$ is the right adjoint to the inclusion $\cC_{\le r} \to \cC$.
\end{itemize}
We write $T_{\ge r}$ for $T_{>r-1}$, and similarly for other functors. This fits into the general set-up of \cite[\S 4.2]{SSglII}, and is analogous to the situation studied in \cite[\S 6]{SSglII}. The inclusion $\cC_{\le r} \subset \cC$ satisfies (Inj) by Corollary~\ref{cor:Inj}. By the general discussion of \S \ref{ss:loccoh}, for any  object $M$ of $\rD^+(\cC)$, we have a canonical exact triangle
\begin{displaymath}
\rR \Gamma_{\le r}(M) \to M \to \rR \Sigma_{>r}(M) \to.
\end{displaymath}
Our first goal is to prove finiteness properties of these functors. The following is the main result in this direction.

\begin{theorem} \label{thm:loccoh}
Let $M \in \rD^b_{\rm fg}(\cC)$. Then $\rR \Sigma_{>r}(M)$ and $\rR \Gamma_{\le r}(M)$ also belong to $\rD^b_{\rm fg}(\cC)$.
\end{theorem}

\begin{proof}
The inductive argument used in \cite[Theorem~6.10]{SSglII} applies here, if we simply use the following lemma in place of \cite[Corollary~6.9]{SSglII}.
\end{proof}

\begin{lemma}
Suppose $M$ is a finitely generated module in $\cC_{\le r+1}$. Then $\rR^i \Sigma_{>r}(M)$ is finitely generated for all $i$, and vanishes for $i \gg 0$.
\end{lemma}

\begin{proof}
Since $M$ belongs to $\cC_{\le r+1}$ it follows from Theorem~\ref{thm:gen-fin-len} that $T_{>r}(M)$ has finite length. By d\'evissage, we assume that $T_{>r}(M)$ is simple. As in the proof of Theorem~\ref{thm:Q-gen}, we have a resolution
\begin{displaymath}
0 \to T_{>r}(M) \to T_{>r}(I^0) \to \cdots \to T_{>r}(I^n) \to 0
\end{displaymath}
where each $I^i$ is a finitely generated $\cQ_{r+1}$-filtered object. By Proposition~\ref{prop:vanish3}, each $I^i$ is derived saturated with respect to $\cC_{\le r}$. We thus see that $\rR \Sigma_{>r}(M)$ is computed by the complex $I^{\bullet}$, which completes the proof.
\end{proof}

As a consequence of the theorem, we obtain a semi-orthogonal decomposition of the derived category. Define $\cD_r$ to be the full subcategory of $\rD^b_{\rm fg}(\cC)$ spanned by objects $M$ for which $\rR\Gamma_{<r}(M)$ and $\rR \Sigma_{>r}(M)$ vanish.

\begin{corollary}
We have a semi-orthogonal deocomposition
\begin{displaymath}
\rD^b_{\rm fg}(\cC) = \langle \cD_0, \ldots, \cD_r \rangle.
\end{displaymath}
\end{corollary}

\begin{proof}
This follows from iterating the decomposition discussed in \S \ref{ss:loccoh}. See also \cite[\S 4.3]{SSglII}.
\end{proof}

We can also give generators for the derived category.

\begin{theorem} \label{thm:dergen2}
The triangulated category $\cD_r$ is generated by the objects $T_{\ge r}(\cQ_r(V))$ with $V$ a simple $\fS_n$-representation (with $n$ variable).
\end{theorem}

\begin{proof}
The category $\cD_r$ is equivalent to $\rD^b_{\rm fg}(\cC_r)$, where $\cC_r=\cC_{\le r}/\cC_{<r}$ \cite[\S 4.3]{SSglII}. Let $\cC'_r$ be the full subcategory of $\cC_r$ spanned by objects of the form $T_{\ge r}(M)$ where $M$ is killed by $\fh_r$. Then every object of $\cC_r$ has a finite length filtration where the graded pieces belong to $\cC'_r$. Moreover, $\cC'_r$ is equivalent to $\Mod_{R/\fh_r}^{\gen}$. The result thus follows from Theorem~\ref{thm:Q-gen}.
\end{proof}

\begin{corollary}
The triangulated category $\rD^b_{\rm fg}(\cC)$ is generated by the objects $\cQ_r(V)$, where $V$ is an irreducible $\fS_n$-representation and $0 \le r \le s$.
\end{corollary}

\subsection{Grothendieck group}

Let $K=\rK(\Mod_{R/\fh_s})$. We now examine the structure of this group. As in \S \ref{ss:gen-K}, $K$ has the structure of a $\Lambda$-module.

\begin{theorem}
We have the following:
\begin{enumerate}
\item The classes $[\cQ_r(V)]$, with $0 \le r \le s$ and $V$ an irreducible $\fS_n$-module, form a $\bZ$-basis of $K$.
\item The classes $[\cP_r(V)]$, with $0 \le r \le s$ and $V$ an irreducible $\fS_n$-module, form a $\bQ$-basis of $\bQ \otimes K$.
\item As a $(\bQ \otimes \Lambda)$-module, $\bQ \otimes K$ is free of rank $s+1$, with basis $[R/\fh_r]$ for $0 \le r \le s$.
\end{enumerate}
\end{theorem}

\begin{proof}
Use notation as in \S \ref{ss:semi-orth}. For $0 \le r \le s$, put $K_{\le r}=\rK(\cC_{\le r})$ and $K_r=\rK(\cC_r)$ where $\cC_r=\cC_{\le r}/\cC_{\le r-1}$. Note that $K_{\le r}$ is naturally identified with $\rK(\Mod_{R/\fh_r})$, and $K_r$ with $\rK(\Mod_{R/\fh_r}^{\gen})$ (see the proof of Theorem~\ref{thm:dergen2}). By Theorem~\ref{thm:loccoh} and the generalities of \S \ref{ss:loccoh}, the natural map $K_{\le r} \to K$ is injective, and $K_{\le r}/K_{\le r-1} \cong K_r$. The result thus follows from our results on $K_r$ in \S \ref{ss:gen-K}.
\end{proof}

\subsection{Krull--Gabriel dimension}

Let $\cA$ be a locally noetherian Grothendieck abelian category. Inductively define full subcategories $\cA_r$, for $r \ge -1$, of $\cA$ as follows. First, $\cA_{-1}=0$. Having defined $\cA_{r-1}$, we define $\cA_r$ to be the full subcategory spanend by objects that become locally of finite length in $\cA/\cA_{r-1}$. We say that $\cA$ has \defn{Krull--Gabriel dimension} $r$ if $\cA=\cA_r$, with $r$ minimal. We now compute this for the categories studied in this paper.

\begin{theorem}
The Krull--Gabriel dimension of the category $\Mod_{R/\fh_s}$ is $s$.
\end{theorem}

\begin{proof}
Put $\cA=\Mod_{R/\fh_s}$, and let $\cA_r$ be as above. We claim that $\cA_r$ coincides with the category $\cC_{\le r}$ spanned by modules locally annihilated by a power of $\fh_r$. This is clear for $r=0$. Suppose now $r \ge 1$ and we know $\cA_{r-1}=\cC_{\le r-1}$. We show $\cA_r=\cC_{\le r}$. Every object of $\cC_{\le r}$ is locally of finite length in $\cA/\cA_{r-1}$ by Theorem~\ref{thm:gen-fin-len}, and so $\cC_{\le r} \subset \cA_r$. Every object of $\cC$ has a (possibly infinite) filtration where the graded pieces are killed by $\fh_r$. It follows that the same is true for objects of $\cA/\cA_{r-1}$. We thus see that the simple objects of $\cA/\cA_{r-1}$ come from objects of $\cA$ that are killed by $\fh_r$. Hence a finite length object of $\cA/\cA_{r-1}$ comes from an object of $\cA$ that is killed by a power of $\fh_r$. Thus $\cA_r \subset \cC_{\le r}$, as required. From this, we see that $\cA=\cA_s$, and $s$ is minimal, proving the theorem.
\end{proof}


\begin{thebibliography}{LNNR2}

\bibitem[AH]{AschenbrennerHillar} Matthias Aschenbrenner, Christopher J.~Hillar. Finite generation of symmetric ideals. \textit{Trans.\ Amer.\ Math.\ Soc.}, {\bf 359} (2007), 5171--5192; erratum, ibid.\ \textbf{361} (2009), pp.~5627--5627. \DOI{10.1090/S0002-9947-07-04116-5} \arxiv{math/0411514}

\bibitem[Coh1]{Cohen} D.~E.~Cohen. On the laws of a metabelian variety. \textit{J.\ Algebra} \textbf{5} (1967), pp.~267--273. \DOI{10.1016/0021-8693(67)90039-7}

\bibitem[Coh2]{Cohen2} D.~E.~Cohen. Closure relations, Buchberger's algorithm, and polynomials in infinitely many variables. In {\it Computation theory and logic}, volume 270 of \textit{Lect.\ Notes Comput.\ Sci.}, pp.\ 78--87, 1987. \DOI{10.1007/3-540-18170-9\_156}

\bibitem[CDD${}^+$]{CDDEF} Christopher H. Chiu, Alessandro Danelon, Jan Draisma, Rob H. Eggermont, Azhar Farooq. Sym-Noetherianity for powers of $\mathbf{GL}$-varieties. \arxiv{2212.05790}

\bibitem[CEF]{fimodule} Thomas Church, Jordan S. Ellenberg, Benson Farb. $\FI$-modules and stability for representations of symmetric groups. \textit{Duke Math.\ J.} \textbf{164} (2015), no.~9, pp.~1833--1910. \DOI{10.1215/00127094-3120274} \arxiv{1204.4533v4}

\bibitem[CEFN]{fimodule2} Thomas Church, Jordan S. Ellenberg, Benson Farb, Rohit Nagpal. $\FI$-modules over Noetherian rings. \textit{Geom.\ Topol.} \textbf{18} (2014) pp.~2951--2984. \DOI{10.2140/gt.2014.18.2951} \arxiv{1210.1854}

\bibitem[Dja]{Djament} Aur\'elien Djament. Des propri\'et\'es de finitude des foncteurs polynomiaux. \textit{Fund.\ Math.} \textbf{233} (2016), pp.~197--256. \DOI{10.4064/fm954-8-2015} \arxiv{1308.4698v6}

\bibitem[Dra]{DraismaNotes} Jan Draisma. Noetherianity up to symmetry. \textit{Combinatorial algebraic geometry}, Lecture Notes in Math.\ \textbf{2108}, Springer, 2014. \arxiv{1310.1705v2}

\bibitem[DE]{DraismaEggermont} Jan Draisma, Rob H. Eggermont. Pl\"ucker varieties and higher secants of Sato's Grassmannian. \textit{J.\ Reine Angew.\ Math.} \textbf{737} (2018), pp.~189--215. \DOI{10.1515/crelle-2015-0035} \arxiv{1402.1667v2}

\bibitem[DK]{DraismaKuttler} Jan Draisma, Jochen Kuttler. Bounded-rank tensors are defined in bounded degree. \textit{Duke Math.\ J.} \textbf{163} (2014), no.~1, pp.~35--63. \DOI{10.1215/00127094-2405170} \arxiv{1103.5336v2}

\bibitem[DEKL]{DEKL} Jan Draisma, Rob H. Eggermont, Robert Krone, Anton Leykin. Noetherianity for infinite-dimensional toric varieties. \textit{Algebra Number Theory} \textbf{9} (2015), no.~8, pp.~1857--1880. \DOI{10.2140/ant.2015.9.1857} \arxiv{1306.0828}

\bibitem[HS]{HillarSullivant} Christopher J.~Hillar, Seth Sullivant. Finite Gr\"obner bases in infinite dimensional polynomial rings and applications. \textit{Adv.\ Math.} \textbf{229} (2012), no.~1, pp.~1--25. \DOI{10.1016/j.aim.2011.08.009} \arxiv{0908.1777}

\bibitem[Gab]{Gabriel} Pierre Gabriel. Des cat\'egories ab\'eliennes. \textit{Bull.\ Soc.\ Math.\ France} \textbf{90} (1962), pp.~323--448. \DOI{10.24033/bsmf.1583}

\bibitem[Gan]{Ganapathy} Karthik Ganapathy. $\mathbf{GL}$-algebras in positive characteristic II: the polynomial ring. \arxiv{2407.13604}

\bibitem[GN]{GunturkunNagel} Sema G\"unt\"urk\"un, Uwe Nagel. Equivariant Hilbert series of monomial orbits. \textit{Proc.\ Amer.\ Math.\ Soc.} \textbf{146} (2018), pp.~2381--2393. \DOI{10.1090/proc/13943} \arxiv{1608.06372}

\bibitem[GS]{increp} Sema G\"unt\"urk\"un, Andrew Snowden. The representation theory of the increasing monoid. \textit{Mem.\ Amer.\ Math.\ Soc.} \textbf{286} (2023), no.~1420, 134 pp. \DOI{10.1090/memo/1420} \arxiv{1812.10242}

\bibitem[KLS]{KLS} Robert Krone, Anton Leykin, Andrew Snowden. Hilbert series of symmetric ideals in infinite polynomial rings via formal languages. \textit{J.\ Algebra} \textbf{485} (2017), pp.~353--362. \DOI{10.1016/j.jalgebra.2017.05.014} \arxiv{1606.07956}

\bibitem[LNNR1]{LNNR} Dinh Van Le, Uwe Nagel, Hop D. Nguyen, Tim R\"omer. Castelnuovo--Mumford regularity up to symmetry. \textit{Int.\ Math.\ Res.\ Not.\ IMRN} \textbf{2021}, no.~14, pp.~11010--11049. \DOI{10.1093/imrn/rnz382} \arxiv{1806.00457}

\bibitem[LNNR2]{LNNR2} Dinh Van Le, Uwe Nagel, Hop D. Nguyen, Tim R\"omer. Codimension and projective dimension up to symmetry. \textit{Math.\ Nachr.} \textbf{293} (2020), no.~2, pp.~346--362. \DOI{10.1002/mana.201800413} \arxiv{1809.06877}

\bibitem[MN]{MarajNagel} Aida Maraj, Uwe Nagel. Shift invariant algebras, Segre products and regular languages. \textit{J.\ Algebra} \textbf{631} (2023), pp.~236--266. \DOI{10.1016/j.jalgebra.2023.04.016} \arxiv{2204.07849}

\bibitem[Nag1]{Nagpal1} Rohit Nagpal. FI-modules and the cohomology of modular $S_n$-representations. \arxiv{1505.04294v1}

\bibitem[Nag2]{Nagpal2} Rohit Nagpal. VI modules in non-describing characteristic, Part II. \textit{J.\ Reine Angew.\ Math} \textbf{781} (2021), pp.~187--205. \DOI{10.1515/crelle-2021-0054} \arxiv{1810.04592}

\bibitem[NR1]{NagelRomer} Uwe Nagel, Tim R\"omer. Equivariant Hilbert series in non-noetherian polynomial rings. \textit{J.\ Algebra} \textbf{486} (2017), pp.~204--245. \DOI{10.1016/j.jalgebra.2017.05.011} \arxiv{1510.02757}

\bibitem[NR2]{NagelRomer2} Uwe Nagel, Tim R\"omer. FI- and OI-modules with varying coefficients. \textit{J.\ Algebra} \text{bf 535} (2019), pp.~286--322. \DOI{10.1016/j.jalgebra.2019.06.029} \arxiv{1710.09247}

\bibitem[NS1]{periodicity} Rohit Nagpal, Andrew Snowden. Periodicity in the cohomology of symmetric groups via divided powers. \textit{Proc.\ London Math.\ Soc.} \textbf{116} (2018), pp.~1244--1268. \DOI{10.1112/plms.12107} \arxiv{1705.10028}

\bibitem[NS2]{svar} Rohit Nagpal, Andrew Snowden. Symmetric subvarieties of infinite affine space. \arxiv{2011.09009}

\bibitem[NS3]{sideals} Rohit Nagpal, Andrew Snowden. Symmetric ideals of the infinite polynomial ring. \arxiv{2107.13027}.

\bibitem[NSS1]{sym2noeth} Rohit Nagpal, Steven Sam, Andrew Snowden. Noetherianity of some degree two twisted commutative algebras. \textit{Selecta Math.\ (N.S.)} \textbf{22} (2016), no.~ ~2, pp.~913--937. \DOI{10.1007/s00029-015-0205-y} \arxiv{1501.06925}

\bibitem[NSS2]{periplectic} Rohit Nagpal, Steven Sam, Andrew Snowden. Noetherianity of some degree two twisted skew-commutative algebras. \textit{Selecta Math.\ (N.S.)} \textbf{25} (2019), article 4. \DOI{10.1007/s00029-019-0461-3} \arxiv{1610.01078}


\bibitem[Sno1]{deltamod} Andrew Snowden. Syzygies of Segre embeddings and $\Delta$-modules. \textit{Duke Math. J.} \textbf{162} (2013), no.~2, pp.~225--277. \DOI{10.1215/00127094-1962767} \arxiv{1006.5248}

\bibitem[Sno2]{msri} Andrew Snowden. Algebraic structures in representation stability, 2019. \\
Available at {\tiny\url{https://websites.umich.edu/~asnowden/msri19/course.pdf}}

\bibitem[SS1]{SSglI} Steven~V Sam, Andrew Snowden. GL-equivariant modules over polynomial rings in infinitely many variables. I. \textit{Trans.\ Amer.\ Math.\ Soc.} \textbf{368} (2016), pp.~1097--1158. \DOI{10.1090/tran/6355} \arxiv{1206.2233v3}

\bibitem[SS2]{infrank} Steven Sam, Andrew Snowden. Stability patterns in representation theory. \textit{Forum Math.\ Sigma} \textbf{3} (2015), e11, 108 pp. \DOI{10.1017/fms.2015.10} \arxiv{1302.5859}

\bibitem[SS3]{SSglII} Steven~V Sam, Andrew Snowden. GL-equivariant modules over polynomial rings in infinitely many variables. II. \textit{Forum Math.\ Sigma} \textbf{7} (2019), e5, 71 pp. \DOI{10.1017/fms.2018.27} \arxiv{1703.04516v2}

\bibitem[SS4]{symc1sp} Steven Sam, Andrew Snowden. $\mathbf{Sp}$-equivariant modules over polynomial rings in infinitely many variables. \textit{Trans.\ Amer.\ Math.\ Soc.} \textbf{375} (2022), pp.~1671--1701. \DOI{10.1090/tran/8496} \arxiv{2002.03243}

\bibitem[SS5]{brauercat1} Steven Sam, Andrew Snowden. The representation theory of Brauer categories I: triangular categories. \textit{Appl.\ Categ.\ Structures} \textbf{30} (2022), pp.~1203--1256. \DOI{10.1007/s10485-022-09689-7} \arxiv{2006.04328}

\bibitem[Yu]{FInmod} Teresa Yu. Parabolic-equivariant modules over polynomial rings in infinitely many variables. \arxiv{2407.02588}

\end{thebibliography}
\end{document}